\documentclass[a4paper,onecolumn,english]{article}

\usepackage[T1]{fontenc}
\usepackage[utf8]{inputenc}
\usepackage{babel}
\usepackage{pdflscape}
\usepackage{multirow}
\usepackage{enumitem}
\usepackage{amsmath}
\usepackage{amsfonts}
\usepackage{amssymb}
\usepackage{amsthm}
\usepackage{thmtools}
\usepackage{dsfont}
\usepackage{bm}
\usepackage{relsize}
\usepackage[left=2.5cm,right=2.5cm,top=2.5cm,bottom=2.5cm]{geometry}
\usepackage[colorlinks=true,
            linkcolor = blue,
            urlcolor  = blue,
            citecolor = red,
			breaklinks]{hyperref}

\allowdisplaybreaks

\numberwithin{equation}{section}

\newcommand{\eps}{\varepsilon}
\newcommand{\mb}[1]{\mathbf{#1}}
\newcommand*{\ccdot}{\kern-.12em\cdot\kern-.12em}
\newcommand{\supp}{\mathrm{supp}}
\newcommand{\warrow}{\overset{w}{\longrightarrow}}
\newcommand{\varrow}{\overset{v}{\longrightarrow}}
\newcommand{\darrow}{\overset{d}{\longrightarrow}}

\newtheoremstyle{slplain}
  {0.4cm}
  {0.4cm}
  {\upshape}
  {}
  {\bfseries}
  {.}
  { }
  {}
  
\newtheoremstyle{itplain}
    {0.4cm}
    {0.4cm}
    {\itshape}
    {}
    {\bfseries}
    {.}
    { }
    {}
  
\declaretheorem[style=slplain,numberwithin=section]{definition}
\declaretheorem[style=slplain,sibling=definition]{example}
\declaretheorem[style=slplain,sibling=definition]{remark}
\declaretheorem[style=slplain,sibling=definition]{assumption}

\declaretheorem[style=itplain,sibling=definition]{theorem}

\declaretheorem[style=itplain,sibling=definition]{proposition}
\declaretheorem[style=itplain,sibling=definition]{lemma}
\declaretheorem[style=itplain,sibling=definition]{corollary}

\AtBeginDocument{
	\addtolength{\abovedisplayskip}{0.125ex}
	\addtolength{\abovedisplayshortskip}{0.125ex}
	\addtolength{\belowdisplayskip}{0.125ex}
	\addtolength{\belowdisplayshortskip}{0.125ex}
}

\renewenvironment{abstract}{%
\noindent\hfill\begin{minipage}{0.92\textwidth}
\rule{\textwidth}{1pt}}
{\par\noindent\rule{\textwidth}{1pt}\end{minipage}\hfill}

\let\OLDthebibliography\thebibliography
\renewcommand\thebibliography[1]{
  \OLDthebibliography{#1}
  \setlength{\parskip}{0pt}
  \setlength{\itemsep}{3pt plus 0.3ex}
}

\title{\bfseries The hyperbolic maximum principle approach to the construction of generalized convolutions\\[0.3cm]}

\author{Rúben Sousa
\thanks{Corresponding author. CMUP, Departamento de Matemática, Faculdade de Ciências, Universidade do Porto, Rua do Campo Alegre 687, 4169-007 Porto, Portugal. Email: \texttt{ruben.sousa@outlook.com}}
\and
Manuel Guerra \thanks{CEMAPRE and ISEG (School of Economics and Management), Universidade de Lisboa, Rua do Quelhas 6, 1200-781 Lisbon, Portugal. Email: \texttt{mguerra@iseg.ulisboa.pt}}
\and
Semyon Yakubovich \thanks{CMUP, Departamento de Matemática, Faculdade de Ciências, Universidade do Porto, Rua do Campo Alegre 687, 4169-007 Porto, Portugal. Email: \texttt{syakubov@fc.up.pt}}
\\[0.3cm]
}
\date{\today\\}

\begin{document}

\maketitle

\begin{abstract}
	 \small
	 \parbox{\linewidth}{\vspace{-2pt}
\begin{center} \bfseries Abstract \vspace{-8pt} \end{center}
	 
	 \-\ \quad We introduce a unified framework for the construction of convolutions and product formulas associated with a general class of regular and singular Sturm-Liouville boundary value problems. Our approach is based on the application of the Sturm-Liouville spectral theory to the study of the associated hyperbolic equation. As a by-product, an existence and uniqueness theorem for degenerate hyperbolic Cauchy problems with initial data at a parabolic line is established.
	 
	 \-\ \quad The mapping properties of convolution operators generated by Sturm-Liouville operators are studied. Analogues of various notions and facts from probabilistic harmonic analysis are developed on the convolution measure algebra. Various examples are presented which show that many known convolution-type operators --- including those associated with the Hankel, Jacobi and index Whittaker integral transforms --- can be constructed using this general approach. \vspace{5pt}
	 
     \-\ \quad \textbf{Keywords:} Generalized convolution, product formula, hyperbolic Cauchy problem, parabolic degeneracy, Sturm-Liouville spectral theory, maximum principle. \vspace{6.5pt}
     }
\end{abstract}

\vspace{8pt}

\begingroup
\let\clearforchapter\relax
\section{Introduction}
\endgroup

Given a Sturm-Liouville operator on an interval of the real line, it is well-known that its eigenfunction expansion gives rise to an integral transform which shares many properties with the ordinary Fourier transform \cite{dunfordschwartz1963,titchmarsh1962}. Since various standard special functions are solutions of Sturm-Liouville equations, the class of integral transforms of Sturm-Liouville type includes, as particular cases, many common integral transforms (Hankel, Kontorovich-Lebedev, Mehler-Fock, Jacobi, Laguerre, etc.).

The Fourier transform lies at the heart of the classical theory of harmonic analysis. This naturally raises a question: \emph{is it possible to generalize the main facts of harmonic analysis to integral transforms of Sturm-Liouville type?}

Starting from the seminal works of Delsarte \cite{delsarte1938} and Levitan \cite{levitan1940} it was noticed that the key ingredient for developing of such a generalized harmonic analysis is the so-called product formula. We say that an indexed family of complex-valued functions $\{w_\lambda\}$ on an interval $I \subset \mathbb{R}$ has a \emph{product formula} if for each $x,y \in I$ there exists a complex Borel measure $\bm{\nu}_{x,y}$ (independent of $\lambda$) such that
\begin{equation} \label{eq:intro_prodform}
w_\lambda(x) \, w_\lambda(y) = \int_I w_\lambda \, d\bm{\nu}_{x,y} \qquad (\lambda \in \Lambda).
\end{equation}
Product formulas naturally lead to generalized convolution operators. To fix ideas, let $\ell(u) = {1 \over r}\bigl[-(pu')' + qu\bigr]$ be a usual Sturm-Liouville differential expression defined on the interval $I$, and let $(\mathcal{F}h)(\lambda) := \int_I h(x) \, w_\lambda(x) \, d\mathrm{m}(x)$ be a Sturm-Liouville type integral transform, where the $w_\lambda$ are solutions of $\ell(w) = \lambda w$ ($\lambda \in \mathbb{C}$). If $\{w_\lambda\}$ has a product formula, then we can define a generalized (Sturm-Liouville type) convolution operator $*$ by
\begin{equation} \label{eq:intro_convdef}
(f * g)(x) := \int_I \biggl( \int_I f \, d\bm{\nu}_{x,y}\biggr)  g(y) \,d\mathrm{m}(y).
\end{equation}
It is not difficult to show that, under reasonable assumptions, the property $\mathcal{F}(f * g) = (\mathcal{F}f) \ccdot (\mathcal{F}g)$ holds for this convolution operator; this means that the analogue of one of the basic identities in harmonic analysis --- the Fourier convolution theorem --- is satisfied by the generalized convolution.

Consider now the associated hyperbolic partial differential equation
\begin{equation} \label{eq:intro_hypPDE}
{1 \over r(x)} \Bigl\{-\partial_x\bigl[p(x) \, \partial_x f(x,y)\bigr] + q(x) f(x,y)\Bigr\} = {1 \over r(y)} \Bigl\{-\partial_y\bigl[p(y) \, \partial_y f(x,y)\bigr] + q(y) f(x,y)\Bigr\}.
\end{equation}
If the kernel of the Sturm-Liouville transform is defined via some initial condition $w_\lambda(a) = 1$, then the product $f(x,y) = w_\lambda(x) w_\lambda(y)$ is a solution of \eqref{eq:intro_hypPDE} satisfying the boundary condition $f(x,a) = w_\lambda(x)$. Studying the properties of the associated hyperbolic equation is therefore a natural strategy for proving the existence of a product formula and extracting information about the measure $\bm{\nu}_{x,y}$.

An especially interesting case is that where $\{\bm{\nu}_{x,y}\}$ turns out to be a family of probability measures (satisfying appropriate continuity assumptions). Indeed, in this case one can show that the convolution \eqref{eq:intro_convdef} gives rise to a Banach algebra structure in the space of finite complex Borel measures in which various probabilistic concepts and properties can be developed in analogy with the classical theory \cite{bloomheyer1994,urbanik1964}. Establishing explicit product formulas, or even proving their existence, has been recognized as a difficult problem \cite{connett1992,chebli1995}. Nevertheless, using the maximum principle for hyperbolic equations \cite{weinberger1956}, it was shown by Levitan \cite{levitan1960} (and, under weakened assumptions, by Chebli \cite{chebli1974} and Zeuner \cite{zeuner1992}) that this probabilistic property of the product formula holds for a general family of Sturm-Liouville differential expressions on $I = [0,\infty)$ of the form $\ell(u) = - {1 \over A} (Au')'$. This family of Sturm-Liouville operators includes, as important particular cases, the generators of the Hankel transform and the (Fourier-)Jacobi transform; these cases are noteworthy due to the fact that the explicit expression for the measure in the product formula can been determined using results from the theory of special functions (see Examples \ref{exam:hankelkingman}--\ref{exam:fourjacobi}).

Various examples show that the probabilistic property of the product formula holds only for a restricted class of Sturm-Liouville operators \cite{litvinov1987,rosler1995}; this is connected with the fact that the hyperbolic maximum principle requires rather strong assumptions on the coefficients. Notwithstanding, the recent work \cite{sousaetal2018a,sousaetal2018b} of the authors on the index Whittaker transform made it apparent that there is room for generalization of the results of \cite{chebli1974,levitan1960,zeuner1992}. In fact, the family of Sturm-Liouville operators considered in these works only includes operators for which the equation \eqref{eq:intro_hypPDE} is uniformly hyperbolic on $[0,\infty)^2$; a consequence of this is that, under their assumptions, the support $\supp(\bm{\nu}_{x,y})$ of the measures in the product formula is always compact. In contrast, the case of the index Whittaker transform provides an example of a product formula whose measures $\bm{\nu}_{x,y}$ have the probabilistic property and satisfy $\supp(\bm{\nu}_{x,y}) = [0,\infty)$ for $x,y > 0$; here the associated hyperbolic equation \eqref{eq:intro_hypPDE} is parabolically degenerate at the boundaries $x=0$ and $y=0$. (The index Whittaker transform is generated by the Sturm-Liouville expression $x^2 u'' + (1+2(1-\alpha) x)u'$ on $I = [0,\infty)$, and its product formula, which is known in closed form, is given in Example \ref{exam:whittaker}.)

The goal of this work is to introduce a unified framework for the construction of Sturm-Liouville type convolution operators associated with possibly degenerate hyperbolic equations. We will consider a Sturm-Liouville differential expression of the form 
\begin{equation} \label{eq:shypPDE_elldiffexpr}
\ell = -{1 \over r} {d \over dx} \Bigl( p \, {d \over dx}\Bigr), \qquad x \in (a,b)
\end{equation}
($-\infty \leq a < b \leq \infty$), where $p$ and $r$ are (real-valued) coefficients such that $p(x), r(x) > 0$ for all $x \in (a,b)$ and $p, p', r$ and $r'$ are locally absolutely continuous on $(a,b)$. Concerning the behavior of the coefficients at the boundaries $x=a$ and $x=b$, we will assume respectively that
\begin{gather}
\label{eq:shypPDE_Lop_leftBC}
\int_a^c \int_y^c {dx \over p(x)} \, r(y) dy < \infty \\
\label{eq:shypPDE_Lop_rightBC}
\int_c^b \int_y^b {dx \over p(x)} \, r(y) dy = \int_c^b \int_c^y {dx \over p(x)} \, r(y) dy = \infty
\end{gather}
where $c \in (a,b)$ is an arbitrary point. These conditions mean that $a$ is a regular or entrance boundary and $b$ is a natural boundary for the operator $\ell$. The notions of regular, entrance and natural boundary refer to the Feller classification of boundaries, which is recalled in Remark \ref{rmk:shypPDE_boundaryclassif}, where we also give some comments on the role of conditions \eqref{eq:shypPDE_Lop_leftBC}--\eqref{eq:shypPDE_Lop_rightBC}.

The point of departure is the study of the Cauchy problem for the possibly degenerate hyperbolic equation ${1 \over r(x)} \partial_x \bigl( p(x) \, \partial_x f(x,y)\bigr) = {1 \over r(y)} \partial_y \bigl( p(y) \, \partial_y f(x,y)\bigr)$. Under the assumption that the product $p(x)r(x)$ of the coefficients of \eqref{eq:shypPDE_elldiffexpr} is an increasing function, we prove an existence and uniqueness theorem for the Cauchy problem which is based on the spectral theory of Sturm-Liouville operators. We then give a sufficient condition for the maximum principle to hold for the hyperbolic equation is given and, as a corollary, the positivity preserving property of the solution of the Cauchy problem is obtained.

Our existence theorem (and the positivity result) covers many hyperbolic equations with initial data on the parabolic line which are outside the scope of the classical theory, and for which the problem of well-posedness of the Cauchy problem was, to the best of our knowledge, open. In fact, given that our results depend heavily on the assumption that the left boundary $a$ is of entrance type (cf.\ Remark \ref{rmk:shypPDE_boundaryclassif}), they indicate that the well-posedness of the degenerate problem with initial line $y=a$ depends on the Feller boundary classification of $\ell$ at the endpoint $a$.

If the maximum principle holds for the hyperbolic equation associated with $\ell$, then the solution of the hyperbolic Cauchy problem can be written as $f(x,y) = \int_{[a,b)} h \, d\bm{\nu}_{x,y}$, where $h(x) = f(x,a)$ is the initial condition and $\{\bm{\nu}_{x,y}\}$ is a family of finite positive Borel measures on $[a,b)$. Formally, this suggests that the product formula \eqref{eq:intro_prodform} should hold for the kernel $w_\lambda$ of the Sturm-Liouville transform. It turns out that \eqref{eq:intro_prodform} indeed holds and that the $\bm{\nu}_{x,y}$ are probability measures, but the proof requires some effort, especially when the Cauchy problem is parabolically degenerate \cite{sousaetalforth}. We then define the generalized convolution by \eqref{eq:intro_convdef}, so that the expected convolution theorem $\mathcal{F}(f * g) = (\mathcal{F}f) \ccdot (\mathcal{F}g)$ holds. Moreover, the Young inequality for the $L_p$-spaces with respect to the weighted measure $r(x)dx$ is valid for the convolution \eqref{eq:intro_convdef}, demonstrating that the mapping properties of the generalized convolution structure resemble those of the ordinary convolution.

A fundamental tool for studying the continuity and mapping properties of the generalized convolution is the extension of the Sturm-Liouville transform to complex measures, defined by $\widehat{\mu}(\lambda) = \int_{[a,b)} \! w_\lambda(x) \mu(dx)$. Actually, if we define the convolution of two Dirac measures by $\delta_x * \delta_y = \bm{\nu}_{x,y}$ and then define the convolution $\mu * \nu$ of two complex measures so that $(\mu,\nu) \mapsto \mu * \nu$ is weakly continuous, then the space $\mathcal{M}_{\mathbb{C}}[a,b)$ of finite complex measures on $[a,b)$ becomes a convolution measure algebra for which the Sturm-Liouville transform is a generalized characteristic function, in the sense that the property $\widehat{\mu * \nu} = \widehat{\mu} \cdot\widehat{\nu}$ holds. The algebra $(\mathcal{M}_{\mathbb{C}}[a,b),*)$ is therefore a natural environment for studying notions from probabilistic harmonic analysis, in particular infinite divisibility, Gaussian-type measures and Lévy-type (additive) stochastic processes. As anticipated above, the study of these concepts leads to analogues of chief results in probability theory such as the Lévy-Khintchine formula or the contraction property of convolution semigroups.

The class of Lévy-type processes with respect to the convolution measure algebra includes the diffusion process generated by the Sturm-Liouville expression $\ell$, as well as many other Markov processes with discontinuous paths. We hope that this work illuminates the role of product formulas and hyperbolic Cauchy problems on a purely probabilistic problem --- that of constructing a class of Lévy-type processes which accommodates a given diffusion process --- and stimulates further research on this topic.

The remaining sections are organized as follows. In Section \ref{sec:prelim}, after introducing the basic properties of the solution of the Sturm-Liouville equation $\ell(w) = \lambda w$, we summarize some key facts from the theory of eigenfunction expansions of Sturm-Liouville operators and from the theory of one-dimensional diffusion processes. Section \ref{sec:hypPDE} is devoted to the hyperbolic Cauchy problem associated with $\ell$: an existence and uniqueness theorem is proved and, under suitable assumptions, it is shown that the unique solution satisfies a weak maximum principle. In Section \ref{sec:transl_conv}, the solution of the hyperbolic Cauchy problem is used to define the generalized convolution of probability measures and the generalized translation of functions; moreover, the Sturm-Liouville transform of finite measures is introduced and an analogue of the Lévy continuity theorem is established, together with some other basic properties. The product formula for the solution of the Sturm-Liouville equation is discussed in Section \ref{sec:prodform}. In Section \ref{sec:Lp_harmonic} we establish the basic properties of the generalized convolution as an operator on weighted $L_p$-spaces. Section \ref{sec:probtheory} explores the probabilistic properties of the convolution, demonstrating that the main concepts and facts from the classical theory of infinitely divisible distributions and convolution semigroups can be developed, in a parallel fashion, in the framework of the generalized convolutions considered here. The concluding Section \ref{sec:examples} presents several examples and shows that various convolutions associated with standard integral transforms constitute particular cases of the general construction presented here. \vspace{10pt}

\section{Preliminaries} \label{sec:prelim}

We use the following standard notations. For a subset $E \subset \mathbb{R}^d$, $\mathrm{C}(E)$ is the space of continuous complex-valued functions on $E$; $\mathrm{C}_\mathrm{b}(E)$, $\mathrm{C}_0(E)$ and $\mathrm{C}_\mathrm{c}(E)$ are, respectively, its subspaces of bounded continuous functions, of continuous functions vanishing at infinity and of continuous functions with compact support; $\mathrm{C}^k(E)$ stands for the subspace of $k$ times continuously differentiable functions. $\mathrm{B}_\mathrm{b}(E)$ is the space of complex-valued bounded and Borel measurable functions. The corresponding spaces of real-valued functions are denoted by $\mathrm{C}(E,\mathbb{R})$, $\mathrm{C}_\mathrm{b}(E,\mathbb{R})$, etc.

$L_p(E;\mu)$ ($1 \leq p \leq \infty$) denotes the Lebesgue space of complex-valued $p$-integrable functions with respect to a given measure $\mu$ on $E$. The space of probability (respectively, finite positive, finite complex) Borel measures on $E$ will be denoted by $\mathcal{P}(E)$ (respectively, $\mathcal{M}_+(E)$, $\mathcal{M}_{\mathbb{C}}(E)$). The total variation of $\mu \in \mathcal{M}_{\mathbb{C}}(E)$ is denoted by $\|\mu\|$, and $\delta_x$ denotes the Dirac measure at a point $x$.

\subsection{Solutions of the Sturm-Liouville equation}

We begin by collecting some properties of the solutions of the Sturm-Liouville equation $\ell(u) = \lambda u$ ($\lambda \in \mathbb{C}$), where $\ell$ is of the form \eqref{eq:shypPDE_elldiffexpr} and satisfies the boundary condition \eqref{eq:shypPDE_Lop_leftBC}. We shall write $f^{[1]} = p f'$ and $\mathfrak{s}(x) = \int_c^x {d\xi \over p(\xi)}$ (this is the so-called \emph{scale function}, cf.\ \cite{borodinsalminen2002}).

If the Sturm-Liouville equation is regular at the left endpoint $a$, it is well-known that there is an entire solution $w_\lambda(x)$ of $\ell(u) = \lambda u$ satisfying the initial conditions $w_\lambda(a) = \cos\theta$, $w_\lambda^{[1]}(a) = \sin\theta$ ($0 \leq \theta < \pi$). When we only require that \eqref{eq:shypPDE_Lop_leftBC} holds (so that $a$ may be an entrance boundary), the following lemma ensures that the same continues to hold for the boundary condition with vanishing derivative ($\theta = 0$):

\begin{lemma} \label{lem:shypPDE_ode_wsol}
For each $\lambda \in \mathbb{C}$, there exists a unique solution $w_\lambda(\cdot)$ of the boundary value problem
\begin{equation} \label{eq:shypPDE_ode_wsol}
\ell(w) = \lambda w \quad (a < x < b), \qquad\;\; w(a) = 1, \qquad\;\; w^{[1]}(a) = 0.
\end{equation}
Moreover, $\lambda \mapsto w_\lambda(x)$ is, for each fixed $x$, an entire function of exponential type.
\end{lemma}

\begin{proof}
The proof is similar to \cite[Lemma 3]{kac1967}, but for completeness we give a sketch here. Let
\begin{equation} \label{eq:shypPDE_wsol_powseries_eta}
\eta_0(x) = 1, \qquad \eta_j(x) = \int_a^x \bigl(\mathfrak{s}(x) - \mathfrak{s}(\xi)\bigr) \eta_{j-1}(\xi) r(\xi) d\xi \quad (j=1,2,\ldots).
\end{equation}
Pick an arbitrary $\beta \in (a,b)$ and define $\mathcal{S}(x) = \int_a^x \bigl(\mathfrak{s}(\beta) - \mathfrak{s}(\xi)\bigr) r(\xi) d\xi$. From the boundary assumption \eqref{eq:shypPDE_Lop_leftBC} it follows that $0 \leq \mathcal{S}(x) \leq \mathcal{S}(\beta) < \infty$ for $x \in (a,\beta]$. Furthermore, it is easy to show (using induction) that $|\eta_j(x)| \leq {1 \over j!} (\mathcal{S}(x))^j$ for all $j$. Therefore, the function
\[
w_\lambda(x) = \sum_{j=0}^\infty (-\lambda)^j \eta_j(x) \qquad (a < x \leq \beta, \; \lambda \in \mathbb{C})
\]
is well-defined as an absolutely convergent series. The estimate
\[
|w_\lambda(x)| \leq \sum_{j=0}^\infty |\lambda|^j {(\mathcal{S}(x))^j \over j!} = e^{|\lambda| \mathcal{S}(x)} \leq e^{|\lambda| \mathcal{S}(\beta)} \qquad (a < x \leq \beta)
\]
shows that $\lambda \mapsto w_\lambda(x)$ is entire and of exponential type. In addition, for $a < x \leq \beta$ we have
\begin{align*}
1 - \lambda \int_a^x {1 \over p(y)} \int_a^y w_\lambda(\xi) \, r(\xi) d\xi\, dy & = 1 - \lambda \int_a^x (\mathfrak{s}(x) - \mathfrak{s}(\xi)) w_\lambda(\xi)\, r(\xi) d\xi \\
& = 1-\lambda \int_a^x (\mathfrak{s}(x) - \mathfrak{s}(\xi)) \biggl( \sum_{j=0}^\infty (-\lambda)^j \eta_j(\xi) \biggr) r(\xi) d\xi \\
& = 1 + \sum_{j=0}^\infty (-\lambda)^{j+1} \int_a^x (\mathfrak{s}(x) - \mathfrak{s}(\xi)) \eta_j(\xi)\, r(\xi) d\xi \\
& = 1 + \sum_{j=0}^\infty (-\lambda)^{j+1} \eta_{j+1}(x) \, = \, w_\lambda(x),
\end{align*}
i.e., $w_\lambda(x)$ satisfies
\[
w_\lambda(x) = 1 - \lambda \int_a^x {1 \over p(y)} \int_a^y w_\lambda(\xi) \, r(\xi)d\xi\, dy
\]
This integral equation is equivalent to \eqref{eq:shypPDE_ode_wsol}, so the proof is complete.
\end{proof}

Throughout this work, $\{a_m\}_{m \in \mathbb{N}}$ will denote a sequence $b > a_1 > a_2 > \ldots$ with $\lim a_m = a$. Next we verify that the solution $w_\lambda$ for the Sturm-Liouville equation on the interval $(a,b)$ is approximated by the corresponding solutions on the intervals $(a_m,b)$:

\begin{lemma} \label{lem:shypPDE_ode_wepslimit}
For $m \in \mathbb{N}$, let $w_{\lambda,m}(x)$ be the unique solution of the boundary value problem
\begin{equation} \label{eq:shypPDE_ode_wsoleps}
\ell(w) = \lambda w \quad (a_m < x < b), \qquad\;\; w(a_m) = 1, \qquad\;\; w^{[1]}(a_m) = 0.
\end{equation}
Then
\[
\lim_{m \to \infty} w_{\lambda,m}(x) = w_\lambda(x) \quad  \text{pointwise for each } a < x < b \text{ and } \lambda \in \mathbb{C}.
\]
\end{lemma}

\begin{proof}
In the same way as in the proof of Lemma \ref{lem:shypPDE_ode_wsol} we can check that the solution of \eqref{eq:shypPDE_ode_wsoleps} is given by
\[
w_{\lambda,m}(x) = \sum_{j=0}^\infty (-\lambda)^j \eta_{j,m}(x) \qquad (a_m < x < b, \; \lambda \in \mathbb{C})
\]
where $\eta_{0,m}(x) = 1$ and $\eta_{j,m}(x) = \int_{a_m}^x \bigl(\mathfrak{s}(x) - \mathfrak{s}(\xi)\bigr) \eta_{j-1,m}(\xi) r(\xi) d\xi$. As before we have $|\eta_{j,m}(x)| \leq {1 \over j!} (\mathcal{S}(x))^j$ for $a_m < x \leq \beta$ (where $\mathcal{S}$ is the function from the proof of Lemma \ref{lem:shypPDE_ode_wsol}). Using this estimate and induction on $j$, it is easy to see that $\eta_{j,m}(x) \to \eta_j(x)$ as $m \to \infty$\, ($a < x \leq \beta$, $j=0,1,\ldots$). Noting that the estimate on $|\eta_{j,m}(x)|$ allows us to take the limit under the summation sign, we conclude that $w_{\lambda,m}(x) \to w_\lambda(x)$ as $m \to \infty$ ($a < x \leq \beta$).
\end{proof}

The following lemma provides a sufficient condition for the solution $w_\lambda(\cdot)$ to be uniformly bounded in the variables $x \in (a,b)$ and $\lambda \geq 0$:

\begin{lemma} \label{lem:shypPDE_wsolbound}
If $x \mapsto p(x)r(x)$ is an increasing function, then the solution of \eqref{eq:shypPDE_ode_wsol} is bounded:
\[
|w_\lambda(x)| \leq 1 \qquad \text{for all } \, a < x < b, \; \lambda \geq 0.
\]
\end{lemma}

\begin{proof}
Let us start by assuming that $p(a)r(a) > 0$. For $\lambda = 0$ the result is trivial because $w_0(x) \equiv 1$. Fix $\lambda > 0$. Multiplying both sides of the differential equation $\ell(w_\lambda) = \lambda w_\lambda$ by $2w_\lambda^{[1]}$, we obtain $-{1 \over pr} [(w_\lambda^{[1]})^2]' = \lambda (w_\lambda^2)'$. Integrating the differential equation and then using integration by parts, we get
\begin{align*}
\lambda\bigl(1-w_\lambda(x)^2\bigr) & = \int_a^x {1 \over p(\xi) r(\xi)} \bigl(w_\lambda^{[1]}(\xi)^2\bigr)' d\xi \\
& = {w_\lambda^{[1]}(x)^2 \over p(x) r(x)} + \int_a^x \bigl(p(\xi)r(\xi)\bigr)' \biggl({w_\lambda^{[1]}(\xi) \over p(\xi)r(\xi)}\biggr)^{\!2} d\xi, \qquad a < x < b
\end{align*}
where we also used the fact that $w_\lambda^{[1]}(a) = 0$ and the assumption that $p(a)r(a) > 0$. The right hand side is nonnegative, because $x \mapsto p(x) r(x)$ is increasing and therefore $(p(\xi)r(\xi))' \geq 0$. Given that $\lambda > 0$, it follows that $1 - w_\lambda(x)^2 \geq 0$, so that $|w_\lambda(x)| \leq 1$.

If $p(a)r(a) = 0$, the above proof can be used to show that the solution of \eqref{eq:shypPDE_ode_wsoleps} is such that $|w_{\lambda,m}(x)| \leq 1$ for all $a < x < b$, $\lambda \geq 0$ and $m \in \mathbb{N}$; then Lemma \ref{lem:shypPDE_ode_wepslimit} yields the desired result.
\end{proof}

\begin{remark} \label{rmk:shypPDE_tildeell}
We shall make extensive use of the fact that the differential expression \eqref{eq:shypPDE_elldiffexpr} can be transformed into the standard form
\[
\widetilde{\ell} = - {1 \over A} {d \over d\xi} \Bigl(A {d \over d\xi} \Bigr) = -{d^2 \over d\xi^2} - {A' \over A} {d \over d\xi}.
\]
This is achieved by setting 
\begin{equation} \label{eq:shypPDE_tildeell_A}
A(\xi) := \sqrt{p(\gamma^{-1}(\xi)) \, r(\gamma^{-1}(\xi))},
\end{equation}
where $\gamma^{-1}$ is the inverse of the increasing function 
\[
\gamma(x) = \int_c^x\! \smash{\sqrt{r(y) \over p(y)}} dy,
\]
$c \in (a,b)$ being a fixed point (if $\smash{\sqrt{r(y) \over p(y)}}$ is integrable near $a$, we may also take $c=a$). Indeed, it is straightforward to check that a given function $\omega_\lambda: (a,b) \to \mathbb{C}$ satisfies $\ell(\omega_\lambda) = \lambda \omega_\lambda$ if and only if $\widetilde{\omega}_\lambda(\xi) := \omega_\lambda(\gamma^{-1}(\xi))$ satisfies $\widetilde{\ell}(\widetilde{\omega}_\lambda) = \lambda \widetilde{\omega}_\lambda$.

It is interesting to note that the assumption of the previous lemma ($x \mapsto p(x) r(x)$ is increasing) is equivalent to requiring that the first-order coefficient $A' \over A$ of the transformed operator $\widetilde{\ell}$ is nonnegative. We also observe that if this assumption holds then we have $\gamma(b) = \infty$ (otherwise the left-hand side integral in \eqref{eq:shypPDE_Lop_rightBC} would be finite, contradicting that $b$ is a natural boundary). We have $\gamma(a) > -\infty$ if $a$ is a regular endpoint (Remark \ref{rmk:shypPDE_boundaryclassif}); if $a$ is entrance, $\gamma(a)$ can be either finite or infinite.
\end{remark}

\subsection{Sturm-Liouville type transforms}

For simplicity, we shall write $L_p(r) := L_p\bigl((a,b); r(x)dx\bigr)$\, ($1 \leq p < \infty$), and the norm of this space will be denoted by $\|\cdot\|_p$.

It follows from the boundary conditions \eqref{eq:shypPDE_Lop_leftBC}--\eqref{eq:shypPDE_Lop_rightBC} that one obtains a self-adjoint realization of $\ell$ in the Hilbert space $L_2(r)$ by imposing the Neumann boundary condition $\lim_{x \downarrow a}u^{[1]}(x) = 0$ at the left endpoint $a$. We state this well-known fact (cf.\ \cite{mckean1956,linetsky2004}) as a lemma:

\begin{lemma}
The operator
\[
\mathcal{L}: \mathcal{D}_\mathcal{L}^{(2)} \subset L_2(r) \longrightarrow L_2(r), \qquad\quad \mathcal{L} u = \ell(u)
\]
where
\begin{equation} \label{eq:shypPDE_Lop_L2domain}
\mathcal{D}_\mathcal{L}^{(2)} := \Bigl\{ u \in L_2(r) \Bigm| u \text{ and } u' \text{ locally abs.\ continuous on } (a,b), \; \ell(u) \in L_2(r), \; \lim_{x \downarrow a} u^{[1]}(x) = 0 \Bigr\}
\end{equation}
is self-adjoint.
\end{lemma}
The self-adjoint realization $\mathcal{L}$ gives rise to an integral transform, which we will call the \emph{$\mathcal{L}$-transform}, given by
\begin{equation} \label{eq:shypPDE_Ltransfdef}
(\mathcal{F} h)(\lambda) := \int_a^b h(x) \, w_\lambda(x) \, r(x) dx \qquad (h \in L_1(r), \; \lambda \geq 0)
\end{equation}
(this is also known as the generalized Fourier transform or the Sturm-Liouville transform). The $\mathcal{L}$-transform is an isometry with an inverse which can be written as an integral with respect to the so-called \emph{spectral measure} $\bm{\rho}_\mathcal{L}$:

\begin{proposition} \label{prop:shypPDE_Ltransf}
There exists a unique locally finite positive Borel measure $\bm{\rho}_\mathcal{L}$ on $\mathbb{R}$ such that the map $h \mapsto \mathcal{F} h$ induces an isometric isomorphism $\mathcal{F}: L_2(r) \longrightarrow L_2(\mathbb{R}; \bm{\rho}_\mathcal{L})$ whose inverse is given by
\[
(\mathcal{F}^{-1} \varphi)(x) = \int_\mathbb{R} \varphi(\lambda) \, w_\lambda(x) \, \bm{\rho}_\mathcal{L}(d\lambda),
\]
the convergence of the latter integral being understood with respect to the norm of $L_2(r)$. The spectral measure $\bm{\rho}_\mathcal{L}$ is supported on $[0,\infty)$. Moreover, the differential operator $\mathcal{L}$ is connected with the transform \eqref{eq:shypPDE_Ltransfdef} via the identity
\begin{equation} \label{eq:shypPDE_Ltransfidentity}
[\mathcal{F} (\mathcal{L} h)] (\lambda) = \lambda \ccdot (\mathcal{F} h)(\lambda), \qquad h \in \mathcal{D}_\mathcal{L}^{(2)}
\end{equation}
and the domain $\mathcal{D}_\mathcal{L}^{(2)}$ defined by \eqref{eq:shypPDE_Lop_L2domain} can be written as
\begin{equation} \label{eq:shypPDE_LtransfidentD2}
\mathcal{D}_\mathcal{L}^{(2)} = \Bigl\{ u \in L_2(r) \Bigm| \lambda \ccdot (\mathcal{F} f)(\lambda) \in L_2\bigl([0,\infty); \bm{\rho}_\mathcal{L}\bigr) \Bigr\}.
\end{equation}
\end{proposition}

\begin{proof}
The existence of a generalized Fourier transform associated with the operator $\mathcal{L}$ is a consequence of the standard Weyl-Titchmarsh-Kodaira theory of eigenfunction expansions of Sturm-Liouville operators (see \cite[Section 3.1]{sousayakubovich2018} and \cite[Section 8]{weidmann1987}).

In the general case the eigenfunction expansion is written in terms of two linearly independent eigenfunctions and a $2 \times 2$ matrix measure. However, from the regular/entrance boundary assumption \eqref{eq:shypPDE_Lop_leftBC} it follows that the function $w_\lambda(x)$ is square-integrable near $x = 0$ with respect to the measure $r(x)dx$; moreover, by Lemma \ref{lem:shypPDE_ode_wsol}, $w_\lambda(x)$ is (for fixed $x$) an entire function of $\lambda$. Therefore, the possibility of writing the expansion in terms only of the eigenfunction $w_\lambda(x)$ follows from the results of \cite[Sections 9 and 10]{eckhardt2013}.
\end{proof}

It is worth pointing out that the transformation of the Sturm-Liouville operator $\ell$ into its standard form $\widetilde{\ell}$ (Remark \ref{rmk:shypPDE_tildeell}) leaves the spectral measure unchanged: indeed, it is easily verified that the operator $\widetilde{\mathcal{L}}: \mathcal{D}_{\widetilde{\mathcal{L}}}^{(2)} \subset L_2(A) \longrightarrow L_2(A)$,\, $\widetilde{\mathcal{L}} u = \widetilde{\ell}(u)$ is unitarily equivalent to the operator $\mathcal{L}$ and, consequently, $\bm{\rho}_{\widetilde{\mathcal{L}}} = \bm{\rho}_\mathcal{L}$.

The following lemma gives a sufficient condition for the inversion integral of the $\mathcal{L}$-transform to be absolutely convergent.

\begin{lemma} \label{lem:shypPDE_Ltransf_D2prop}
\textbf{(a)} For each $\mu \in \mathbb{C} \setminus\mathbb{R}$, the integrals
\begin{equation} \label{eq:shypPDE_Lresolv_specunif}
\int_{[0,\infty)} {w_\lambda(x)\, w_\lambda(y) \over |\lambda - \mu|^2} \bm{\rho}_\mathcal{L}(d\lambda) \qquad\; \text{and} \qquad\; \int_{[0,\infty)} {w_\lambda^{[1]}(x)\, w_\lambda^{[1]}(y) \over |\lambda - \mu|^2} \bm{\rho}_\mathcal{L}(d\lambda)
\end{equation}
converge uniformly on compact squares in $(a,b)^2$. \\[-8pt]

\textbf{(b)} If $h \in \mathcal{D}_{\mathcal{L}}^{(2)}$, then 
\begin{align} 
\label{eq:shypPDE_Ltransf_D2prop}
h(x) & = \int_{[0,\infty)}\! (\mathcal{F}h)(\lambda) \, w_\lambda(x) \, \bm{\rho}_\mathcal{L}(d\lambda)\\
\label{eq:shypPDE_Ltransf_D2propderiv}
h^{[1]}(x) & = \int_{[0,\infty)}\! (\mathcal{F}h)(\lambda) \, w_\lambda^{[1]}(x) \, \bm{\rho}_\mathcal{L}(d\lambda)
\end{align}
where the right-hand side integrals converge absolutely and uniformly on compact subsets of $(a,b)$.
\end{lemma}

\begin{proof}
\textbf{(a)} By \cite[Lemma 10.6]{eckhardt2013} and \cite[p.\ 229]{teschl2014},
\[
\int_{[0,\infty)} {w_\lambda(x) w_\lambda(y) \over |\lambda - \mu|^2} \bm{\rho}_\mathcal{L}(d\lambda) =  \int_a^b G(x,\xi,\mu)G(y,\xi,\mu)\, r(\xi) d\xi = {1 \over \mathrm{Im}(\mu)}\, \mathrm{Im}\bigl(G(x,y,\mu)\bigr)
\]
where $G(x,y,\mu)$ is the resolvent kernel (or Green function) of the operator $(\mathcal{L}, \mathcal{D}_\mathcal{L}^{(2)})$. Moreover, according to \cite[Theorems 8.3 and 9.6]{eckhardt2013}, the resolvent kernel is given by
\[
G(x,y,\mu) = \begin{cases}
w_\mu(x) \vartheta_\mu(y), & x < y \\
w_\mu(y) \vartheta_\mu(x), & x \geq y
\end{cases}
\]
where $\vartheta_\lambda(\cdot)$ is a solution of $\ell(u) = \lambda u$ which is square-integrable near $\infty$ with respect to the measure $r(x)dx$ and verifies the identity $w_\lambda(x) \vartheta_\lambda^{[1]}(x) - w_\lambda^{[1]}(x) \vartheta_\lambda(x) \equiv 1$. It is easily seen (cf.\ \cite[p.\ 125]{naimark1968}) that the functions $\mathrm{Im}\bigl(G(x,y,\mu)\bigr)$ and $\partial_x^{[1]} \partial_y^{[1]} \mathrm{Im}\bigl(G(x,y,\mu)\bigr)$ are continuous in $0 < x,y < \infty$. Essentially the same proof as that of \cite[Corollary 3]{naimark1968} now yields that
\[
\int_{[0,\infty)} {w_\lambda^{[1]}(x) \, w_\lambda^{[1]}(y) \over |\lambda - \mu|^2} \bm{\rho}_\mathcal{L}(d\lambda) = {1 \over \mathrm{Im}(\mu)}\, \partial_x^{[1]} \partial_y^{[1]} \mathrm{Im}\bigl(G(x,y,\mu)\bigr)
\]
and that the integrals \eqref{eq:shypPDE_Lresolv_specunif} converge uniformly for $x,y$ in compacts.
\\[-8pt]

\textbf{(b)} By Proposition \ref{prop:shypPDE_Ltransf} and the classical theorem on differentiation under the integral sign for Riemann-Stieltjes integrals, to prove \eqref{eq:shypPDE_Ltransf_D2prop}--\eqref{eq:shypPDE_Ltransf_D2propderiv} it only remains to justify the absolute and uniform convergence of the integrals in the right-hand sides.

Recall from Proposition \ref{prop:shypPDE_Ltransf} that the condition $h \in \mathcal{D}_\mathcal{L}^{(2)}$ implies that $\mathcal{F}h \in L_2\bigl([0,\infty); \bm{\rho}_\mathcal{L}\bigr)$ and also $\lambda \,(\mathcal{F}h)(\lambda) \in L_2\bigl([0,\infty); \bm{\rho}_\mathcal{L}\bigr)$. As a consequence, we obtain
\begin{align*}
& \int_{[0,\infty)} \bigl|(\mathcal{F}h)(\lambda) w_\lambda(x)\bigr| \bm{\rho}_\mathcal{L}(d\lambda) \\
& \qquad\qquad \leq \!\int_{[0,\infty)} \! \lambda \, \bigl|(\mathcal{F}h)(\lambda)\bigr| \biggl|{w_\lambda(x) \over \lambda + i}\biggr| \bm{\rho}_\mathcal{L}(d\lambda) + \! \int_{[0,\infty)} \bigl|(\mathcal{F}h)(\lambda)\bigr| \biggl| {w_\lambda(x) \over \lambda + i} \biggr| \bm{\rho}_\mathcal{L}(d\lambda) \\
& \qquad\qquad \leq \bigl(\|\lambda \, (\mathcal{F}h)(\lambda)\|_\rho + \|(\mathcal{F}h)(\lambda)\|_\rho\bigr) \biggl\| {w_\lambda(x) \over \lambda + i} \biggr\|_\rho \\
& \qquad\qquad < \infty
\end{align*}
where $\| \cdot \|_\rho$ denotes the norm of the space $L_2\bigl(\mathbb{R}; \bm{\rho}_\mathcal{L}\bigr)$, and similarly
\[
\int_{[0,\infty)}\! \bigl|(\mathcal{F}h)(\lambda) \, w_\lambda^{[1]}(x) \bigr| \bm{\rho}_\mathcal{L}(d\lambda) \leq \bigl(\|\lambda \, (\mathcal{F}h)(\lambda)\|_\rho + \|(\mathcal{F}h)(\lambda)\|_\rho\bigr) \biggl\| {w_\lambda^{[1]}(x) \over \lambda + i} \biggr\|_\rho < \infty.
\]
We know from part (a) that the integrals which define $\bigl\| {w_\lambda(x) \over \lambda + i} \bigr\|_\rho$ and $\bigl\| {w_\lambda^{[1]}(x) \over \lambda + i} \bigr\|_\rho$ converge uniformly, hence the integrals in \eqref{eq:shypPDE_Ltransf_D2prop}--\eqref{eq:shypPDE_Ltransf_D2propderiv} converge absolutely and uniformly for $x$ in compact subsets.
\end{proof}

\subsection{Diffusion processes}

In what follows we write $P_{x_0}$ for the distribution of a given time-homogeneous Markov process started at the point $x_0$ and $\mathbb{E}_{x_0}$ for the associated expectation operator.

By an \emph{irreducible diffusion process} $X$ on an interval $I \subset \mathbb{R}$ we mean a continuous strong Markov process $\{X_t\}_{t \geq 0}$ with state space $I$ and such that
\[
P_x(\tau_y < \infty) > 0 \; \text{ for any } x \in \mathrm{int} \, I \text{ and } y \in I, \quad\;\; \text{ where } \tau_y = \inf\{t \geq 0 \mid X_t = y\}.
\]
The \emph{resolvent} $\{\mathcal{R}_\eta\}_{\eta > 0}$ of such a diffusion (or of a general Feller process) $X$ is defined by $\mathcal{R}_\eta u = \int_0^\infty e^{-\eta t} \mathcal{P}_t u \, dt$,\, $u \in \mathrm{C}_\mathrm{b}(I,\mathbb{R})$, where $(\mathcal{P}_t u)(x) = \mathbb{E}_x[u(X_t)]$ is the transition semigroup of the process $X$. The \emph{$\mathrm{C}_\mathrm{b}$-generator} $(\mathcal{G}, \mathcal{D}(\mathcal{G}))$ of $X$ is the operator with domain $\mathcal{D}(\mathcal{G}) = \mathcal{R}_\eta\bigl(\mathrm{C}_\mathrm{b}(I,\mathbb{R})\bigr)$ ($\eta > 0$) and defined by 
\[
(\mathcal{G} u)(x) = \eta u(x) - g(x) \qquad \text{ for } u = \mathcal{R}_\eta g, \; g \in \mathrm{C}_\mathrm{b}(I,\mathbb{R}), \; x \in I
\]
($\mathcal{G}$ is independent of $\eta$, cf.\ \cite[p.\ 295]{fukushima2014}). A \emph{Feller semigroup} is a family $\{T_t\}_{t \geq 0}$ of operators $T_t: \mathrm{C}_\mathrm{b}(I,\mathbb{R}) \longrightarrow \mathrm{C}_\mathrm{b}(I,\mathbb{R})$ satisfying
\begin{enumerate}[itemsep=0pt,topsep=4pt]
\item[\textbf{\itshape(i)}] $T_t T_s = T_{t+s}$ for all $t, s \geq 0$;
\item[\textbf{\itshape(ii)}] $T_t \bigl(\mathrm{C}_0(I,\mathbb{R})\bigr) \subset \mathrm{C}_0(I,\mathbb{R})$ for all $t \geq 0$;
\item[\textbf{\itshape(iii)}] If $h \in \mathrm{C}_\mathrm{b}(I,\mathbb{R})$ satisfies $0 \leq h \leq 1$, then $0 \leq T_t h \leq 1$;
\item[\textbf{\itshape(iv)}] $\lim_{t \downarrow 0} \|T_t h - h\|_\infty = 0$ for each $h \in \mathrm{C}_0(I,\mathbb{R})$.
\end{enumerate}
The Feller semigroup is said to be \emph{conservative} if $T_t \mathds{1} = \mathds{1}$ (here $\mathds{1}$ denotes the function identically equal to one). A \emph{Feller process} is a time-homogeneous Markov process $\{X_t\}_{t \geq 0}$ whose transition semigroup is a Feller semigroup. For further background on the theory of Markov diffusion processes and Feller semigroups, we refer to \cite{borodinsalminen2002} and references therein.

We now recall a known fact from the theory of (one-dimensional) diffusion processes, namely that the negative of the Sturm-Liouville differential operator \eqref{eq:shypPDE_elldiffexpr} generates a diffusion process which is conservative and has the Feller property. The proof can be found on \cite[Sections 4 and 6]{fukushima2014} (see also \cite[Section II.5]{mandl1968}).

\begin{lemma} \label{lem:shypPDE_Lb_fellergen}
The operator 
\[
\mathcal{L}^\mathrm{(b)}: \mathcal{D}_{\mathcal{L}}^{(\mathrm{b})} \subset \mathrm{C}_\mathrm{b}([a,b),\mathbb{R})  \longrightarrow \mathrm{C}_\mathrm{b}([a,b),\mathbb{R}), \qquad\quad \mathcal{L}^\mathrm{(b)} u = -\ell(u)
\]
with domain
\[
\mathcal{D}_\mathcal{L}^{(\mathrm{b})} = \bigl\{ u \in \mathrm{C}_\mathrm{b}([a,b),\mathbb{R}) \bigm| \ell(u) \in \mathrm{C}_\mathrm{b}([a,b),\mathbb{R}),\, \lim_{x \downarrow a} u^{[1]}(x) = 0 \bigr\} 
\]
is the $\mathrm{C}_\mathrm{b}$-generator of a one-dimensional irreducible diffusion process $X = \{X_t\}_{t \geq 0}$ with state space $[a,b)$ whose transition semigroup defines a conservative Feller semigroup on $\mathrm{C}_0([a,b),\mathbb{R})$.
\end{lemma}

The transition probabilities of the one-dimensional diffusion process from the previous lemma admits an explicit representation as the inverse $\mathcal{L}$-transform of the function $e^{-t\lambda} w_\lambda(x)$:

\begin{lemma} \label{lem:shypPDE_Lb_diffusion_tpdf}
The transition semigroup admits the representation
\[
(\mathcal{P}_t u)(x) = \int_a^b h(y)\, p(t,x,y)\, r(y)dy \qquad (h \in \mathrm{B}_\mathrm{b}\bigl([a,b),\mathbb{R}\bigr), \; t > 0, \; a < x < b)
\]
where $p(t,x,y)$ is a nonnegative function which is called the \emph{fundamental solution} for the parabolic equation ${\partial u \over \partial t} = -\ell_x u$ (the subscript indicates the variable in which the operator $\ell$ acts). The fundamental solution and its derivatives are explicitly given by
\begin{align*}
(\partial_t^n p)(t,x,y) & = \int_{[0,\infty)} \lambda^n e^{-t\lambda} \, w_\lambda(x) \, w_\lambda(y)\, \bm{\rho}_\mathcal{L}(d\lambda) \\
(\partial_x^{[1]} \partial_t^n p)(t,x,y) & = \int_{[0,\infty)} \lambda^n e^{-t\lambda} \, w_\lambda^{[1]}(x) \, w_\lambda(y)\, \bm{\rho}_\mathcal{L}(d\lambda)
\end{align*}
for $n \in \mathbb{N}_0$, where, for fixed $t > 0$, the integrals converge absolutely and uniformly on compact squares of $(a,b) \times (a,b)$.
\end{lemma}

\begin{proof}
These assertions are a consequence of the results of \cite[Sections 2--3]{linetsky2004} and \cite[Section 4]{mckean1956}.
\end{proof}

We mention also that another consequence of the results of \cite[Section 3]{linetsky2004} is that for $h \in L_2(r)$ the expectation of $h(X_t)$ can be written in terms of the $\mathcal{L}$-transform as
\[
\mathbb{E}_x[h(X_t)] = \int_{[0,\infty)} e^{-t\lambda} w_\lambda(x) \, (\mathcal{F} h)(\lambda)\, \bm{\rho}_\mathcal{L}(d\lambda) \qquad t > 0, \; a < x < b
\]
where the integral converges with respect to the norm of $L_2(r)$. 

\begin{remark} \label{rmk:shypPDE_boundaryclassif}
Let $X$ be a one-dimensional diffusion process on an interval with endpoints $a$ and $b$, whose $\mathrm{C}_\mathrm{b}$-generator is of the form \eqref{eq:shypPDE_elldiffexpr}. Let 
\[
I_a = \int_a^c \int_a^y {dx \over p(x)} \, r(y) dy, \qquad J_a = \int_a^c \int_y^c {dx \over p(x)} \, r(y) dy
\]
According to \emph{Feller's boundary classification} for the diffusion $X$, the left endpoint $a$ is called\\[4pt]
\begin{minipage}{\linewidth}
\centering
\begin{minipage}{0.4\linewidth}
\centering
\begin{tabular}{llll}
\emph{regular} & if & $I_a < \infty$, & $\!\!\!J_a < \infty$; \\
\emph{exit} & if & $I_a < \infty$, & $\!\!\!J_a = \infty$;
\end{tabular}
\end{minipage}
\begin{minipage}{0.4\linewidth}
\centering
\begin{tabular}{llll}
\emph{entrance} & if & $I_a = \infty$, & $\!\!\!J_a < \infty$; \\
\emph{natural} & if & $I_a = \infty$, & $\!\!\!J_a = \infty$.
\end{tabular}
\end{minipage} \vspace{4pt}
\end{minipage}
(the right endpoint is classified in a similar way).

The probabilistic meaning of this classification is the following \cite[Chapter II]{borodinsalminen2002}: an irreducible diffusion can be started from the boundary $a$ if and only if $a$ is regular or entrance; the boundary $a$ is reached from $x_0 \in (a,b)$ with positive probability by an irreducible diffusion if and only if $a$ is regular or exit.

Our standing assumption \eqref{eq:shypPDE_Lop_leftBC} on the coefficients of the Sturm-Liouville operator means that $a$ is a regular or an entrance boundary for the diffusion process $X$ generated by $\ell$. It is clear from the preceding remarks that Lemma \ref{lem:shypPDE_Lb_fellergen} relies crucially on this assumption. The same is true for some of the results of the previous subsections: in fact, Lemma \ref{lem:shypPDE_ode_wsol} fails if $a$ is exit or natural \cite[Sections 5.13--5.14]{ito2006}, and the boundary conditions defining $\mathcal{D}_\mathcal{L}^{(2)}$ differ from those in \eqref{eq:shypPDE_Lop_L2domain} when $a$ is exit or natural \cite{mckean1956}. In turn, the assumption \eqref{eq:shypPDE_Lop_rightBC} means that $b$ is a natural boundary for the diffusion $X$. Since one can show that \eqref{eq:shypPDE_Lop_rightBC} is automatically satisfied whenever Assumption \ref{asmp:shypPDE_SLhyperg} below holds \cite[Proposition 3.5]{sousaetalforth}, this boundary assumption at $b$ yields no loss of generality on our results concerning product formulas and generalized convolutions.
\end{remark}

\section{The hyperbolic equation $\ell_x f = \ell_y f$} \label{sec:hypPDE}

In this section we investigate the (possibly degenerate) hyperbolic Cauchy problem
\begin{equation} \label{eq:shypPDE_Lcauchy}
(\ell_x f)(x,y) = (\ell_y f)(x,y) \quad\; (x,y \in (a,b)), \qquad\quad
f(x,a) = h(x), \qquad\quad
(\partial_y^{[1]}\!f)(x,a) = 0
\end{equation}
where $\partial_{\,}^{[1]} u = pu'$, $\ell$ is the Sturm-Liouville operator \eqref{eq:shypPDE_elldiffexpr}, and the subscripts indicate the variable in which the operators act.

Since $\ell_y - \ell_x = {p(x) \over r(x)} {\partial^2 \over \partial x^2} - {p(y) \over r(y)} {\partial^2 \over \partial y^2} + \text{lower order terms}$, the equation $\ell_x f = \ell_y f$ is hyperbolic at the line $y=a$ if ${p(a) \over r(a)} > 0$; otherwise, the initial conditions of the Cauchy problem are given at a line of parabolic degeneracy. If $\gamma(a) = -\int_a^c\! \sqrt{r(y) \over p(y)} dy > -\infty$, then we can remove the degeneracy via the change of variables $x = \gamma(\xi)$, $y = \gamma(\zeta)$ (cf.\ Remark \ref{rmk:shypPDE_tildeell}), through which the partial differential equation is transformed to the standard form $\widetilde{\ell}_\xi u = \widetilde{\ell}_\zeta u$, with initial condition at the line $\zeta = \gamma(a)$. In the case $\gamma(a) = -\infty$, the standard form of the equation is also parabolically degenerate in the sense that its initial line is $\zeta = -\infty$.

\subsection{Existence and uniqueness of solution}

We start by proving a result which not only assures the existence of solution for  Cauchy problems with well-behaved initial conditions but also provides an explicit representation for the solution as an inverse $\mathcal{L}$-transform:

\begin{theorem}[Existence of solution] \label{thm:shypPDE_Lexistence}
Suppose that $x \mapsto p(x)r(x)$ is an increasing function. If $h \in \mathcal{D}_{\mathcal{L}}^{(2)\!}$ and\, $\ell(h) \in \mathcal{D}_{\mathcal{L}}^{(2)\!}$, then the function
\begin{equation} \label{eq:shypPDE_Lexistence}
f_h(x,y) := \int_{[0,\infty)\!} w_\lambda(x) \, w_\lambda(y) \, (\mathcal{F} h)(\lambda) \, \bm{\rho}_{\mathcal{L}}(d\lambda)
\end{equation}
solves the Cauchy problem \eqref{eq:shypPDE_Lcauchy}.
\end{theorem}

For ease of notation, unless necessary we drop the dependence in $h$ and denote \eqref{eq:shypPDE_Lexistence} by $f(x,y)$.

\begin{proof}
Let us begin by justifying that $\ell_x f$ can be computed via differentiation under the integral sign. It follows from \eqref{eq:shypPDE_ode_wsol} that $w_\lambda^{[1]}(x) = - \lambda \int_a^x w_\lambda(\xi) \,r(\xi) d\xi$ and therefore (by Lemma \ref{lem:shypPDE_wsolbound}) $|w_\lambda^{[1]}(x)| \leq \lambda \int_a^x r(\xi) d\xi$. Hence
\begin{equation} \label{eq:shypPDE_Lexistence_solineq}
\int_{[0,\infty)\!} \bigl| (\mathcal{F} h)(\lambda) \, w_\lambda^{[1]}(x) \, w_\lambda(y)\bigr| \bm{\rho}_{\mathcal{L}}(d\lambda) \leq \int_a^x r(\xi) d\xi \ccdot \int_{[0,\infty)\!} \lambda \, \bigl| (\mathcal{F} h)(\lambda)\, w_\lambda(y) \bigr| \bm{\rho}_{\mathcal{L}}(d\lambda) < \infty,
\end{equation}
where the convergence (which is uniform on compacts) follows from \eqref{eq:shypPDE_Ltransfidentity} and Lemma \ref{lem:shypPDE_Ltransf_D2prop}(b). From the convergence of the differentiated integral we conclude that $\partial_x^{[1]}\!f(x,y) = \int_{[0,\infty)\!} (\mathcal{F} h)(\lambda) \, w_\lambda^{[1]}(x) \, w_\lambda(y) \, \bm{\rho}_{\mathcal{L}}(d\lambda)$. Since $(\ell w_\lambda)(x) = \lambda w_\lambda(x)$, in the same way we check that $\int_{[0,\infty)} (\mathcal{F} h)(\lambda) \, (\ell w_\lambda)(x)\, w_\lambda(y)\, \bm{\rho}_{\mathcal{L}}(d\lambda)$ converges absolutely and uniformly on compacts and is therefore equal to $(\ell_x f)(x,y)$. Consequently,
\begin{equation} \label{eq:shypPDE_Lexistence_ellrepr}
(\ell_x f)(x,y) = (\ell_y f)(x,y) = \int_{[0,\infty)\!} \lambda\, (\mathcal{F} h)(\lambda) \, w_\lambda(x) \, w_\lambda(y) \, \bm{\rho}_{\mathcal{L}}(d\lambda).
\end{equation}
Concerning the boundary conditions, Lemma \ref{lem:shypPDE_Ltransf_D2prop}(b) together with the fact that $w_\lambda(a) = 1$ imply that $f(x,a) = h(x)$, and from \eqref{eq:shypPDE_Lexistence_solineq} we easily see that $\lim_{y \downarrow a} \partial_y^{[1]}\!f(x,y) = 0$. This shows that $f$ is a solution of the Cauchy problem \eqref{eq:shypPDE_Lcauchy}.
\end{proof}

Under the assumptions of the theorem, the solution \eqref{eq:shypPDE_Lexistence} of the hyperbolic Cauchy problem satisfies
\begin{align}
\label{eq:shypPDE_uniq_cond0} f(\cdot, y) \in \mathcal{D}_\mathcal{L}^{(2)} & \qquad \text{for all } \, a < y < b, \\[2pt]
\label{eq:shypPDE_uniq_cond1} \mathcal{F}[\ell_y f(\cdot,y)](\lambda) = \ell_y [\mathcal{F}f(\cdot,y)](\lambda) & \qquad \text{for all } \, a < y < b, \\[2pt]
\label{eq:shypPDE_uniq_cond2} \lim_{y \downarrow a} [\mathcal{F}f(\cdot,y)](\lambda) = (\mathcal{F}h)(\lambda), & \\
\label{eq:shypPDE_uniq_cond3} \lim_{y \downarrow a} \partial_y^{[1]\!} \mathcal{F}[f(\cdot,y)](\lambda) = 0. &
\end{align}
Indeed, by Proposition \ref{prop:shypPDE_Ltransf} we have $[\mathcal{F}f(\cdot,y)](\lambda) = (\mathcal{F}h)(\lambda) \, w_\lambda(y)$ for all $\lambda \in \supp(\bm{\rho}_\mathcal{L})$ and $a < y < b$. Since $h \in \mathcal{D}_{\mathcal{L}}^{(2)\!}$ and $|w_\lambda(\cdot)| \leq 1$ (Lemma \ref{lem:shypPDE_wsolbound}), it is clear from \eqref{eq:shypPDE_LtransfidentD2} that $f(x,y)$ satisfies \eqref{eq:shypPDE_uniq_cond0}. Moreover, it follows from \eqref{eq:shypPDE_Lexistence_ellrepr} that $\mathcal{F}[\ell_y f_j(\cdot,y)](\lambda) = \lambda \, (\mathcal{F}h)(\lambda) \, w_\lambda(y) = \ell_y [\mathcal{F}f_j(\cdot,y)](\lambda)$, hence \eqref{eq:shypPDE_uniq_cond1} holds. The properties \eqref{eq:shypPDE_uniq_cond2}--\eqref{eq:shypPDE_uniq_cond3} follow immediately from Lemma \ref{lem:shypPDE_ode_wsol}.

Next we show that the solution from the above existence theorem is the unique solution satisfying the conditions \eqref{eq:shypPDE_uniq_cond0}--\eqref{eq:shypPDE_uniq_cond3}:

\begin{theorem}[Uniqueness]
Let $h \in \mathcal{D}_{\mathcal{L}}^{(2)}$. Let $f_1, f_2 \in \mathrm{C}^2\bigl((a,b)^2\bigr)$ be two solutions of $(\ell_x f)(x,y) = (\ell_y f)(x,y)$. For $f \in \{f_1,f_2\}$, suppose that \eqref{eq:shypPDE_uniq_cond0} holds and that there exists a zero $\bm{\rho}_\mathcal{L}$-measure set $\Lambda_0 \subset [0,\infty)$ such that \eqref{eq:shypPDE_uniq_cond1}--\eqref{eq:shypPDE_uniq_cond3} hold for each $\lambda \in [0,\infty) \setminus \Lambda_0$. Then 
\begin{equation} \label{eq:shypPDE_uniq}
f_1(x,y) \equiv f_2(x,y) \qquad \text{ for all } \; x,y \in (a,b).
\end{equation}
\end{theorem}

\begin{proof}
Fix $\lambda \in [0,\infty) \setminus \Lambda_0$ and let $\Psi_{\!j}(y,\lambda) := [\mathcal{F}f_j(\cdot,y)](\lambda)$. We have
\[
\ell_y \Psi_{\!j}(y,\lambda) = \mathcal{F}[\ell_y f_j(\cdot,y)](\lambda) = \mathcal{F}[\ell_x f_j(\cdot,y)](\lambda) = \lambda \Psi_{\!j}(y,\lambda), \qquad a < y < b
\]
where the first equality is due to \eqref{eq:shypPDE_uniq_cond1} and the last step follows from \eqref{eq:shypPDE_Ltransfidentity}. Moreover, 
\[
\lim_{y \downarrow a}\Psi_{\!j}(y,\lambda) = (\mathcal{F}h)(\lambda) \quad \text{ and } \quad \lim_{y \downarrow a} \partial_y^{[1]}\Psi_{\!j}(y,\lambda) = 0
\]
by \eqref{eq:shypPDE_uniq_cond2} and \eqref{eq:shypPDE_uniq_cond3}, respectively. It thus follows from Lemma \ref{lem:shypPDE_ode_wsol} that
\[
[\mathcal{F}f_j(\cdot,y)](\lambda) = \Psi_{\!j}(y,\lambda) = (\mathcal{F}h)(\lambda)\, w_\lambda(y), \qquad a < y < b.
\]
This equality holds for $\bm{\rho}_\mathcal{L}$-almost every $\lambda$, so the isometric property of $\mathcal{F}$ gives $f_1(\cdot,y) = f_2(\cdot,y)$ Lebesgue-almost everywhere; since the $f_j$ are continuous, we conclude that \eqref{eq:shypPDE_uniq} holds.
\end{proof}

We emphasize that the two previous propositions, in particular, ensure that there exists a unique solution for the Cauchy problem \eqref{eq:shypPDE_Lcauchy} with initial condition
\[
h \in \mathrm{C}_{\mathrm{c},0}^4 := \Bigl\{u \in \mathrm{C}_\mathrm{c}^4[a,b) \Bigm| \ell(u), \ell^2(u) \in \mathrm{C}_\mathrm{c}[a,b), \;\: \lim_{x \downarrow a} u^{[1]}(x) = \lim_{x \downarrow a} [\ell(u)]^{[1]}(x)= 0 \Bigr\}
\]
(clearly, if $h \in \mathrm{C}_{\mathrm{c},0}^4$ then $h, \ell(h) \in \mathcal{D}_{\mathcal{L}}^{(2)}$).

If the hyperbolic equation $\ell_x f = \ell_y f$ (or the transformed equation $\widetilde{\ell}_\xi u = \widetilde{\ell}_y u$) is uniformly hyperbolic, the existence and uniqueness of solution for this Cauchy problem is a standard result which follows from the classical theory of hyperbolic problems in two variables (see e.g.\ \cite[Chapter V]{couranthilbert1962}); in fact, the existence and uniqueness holds under much weaker restrictions on the initial condition. However, our existence and uniqueness result becomes nontrivial in the presence of a (non-removable) parabolic degeneracy at the initial line. 

Indeed, even though many authors have addressed Cauchy problems for degenerate hyperbolic equations in two variables, most studies are restricted to equations where the ${\partial^2 \over \partial x^2}$ term vanishes at an initial line $y=y_0$ (we refer to \cite[\S2.3]{bitsadze1964}, \cite[Section 5.4]{radkevic2009} and references therein). Much less is known for hyperbolic equations whose ${\partial^2 \over \partial y^2}$ term vanishes at the same initial line: it is known that the Cauchy problem is, in general, not well-posed, and the relevance of determining sufficient conditions for its well-posedness has long been pointed out \cite[\S2.4]{bitsadze1964}, but as far as we are aware little progress has been made on this problem (for related work see \cite{mamadaliev2000}). The application of spectral techniques to hyperbolic Cauchy problems associated with Sturm-Liouville operators is by no means new, see e.g.\ \cite{chebli1974,carroll1985} and references therein; however, it seems that such techniques had never been applied to degenerate cases.

It is helpful to know that an existence theorem analogous to Theorem \ref{thm:shypPDE_Lexistence} holds when the initial line is shifted away from the degeneracy, because this will allow us to justify that the solution of the degenerate Cauchy problem is the pointwise limit of solutions of nondegenerate problems.

\begin{proposition} \label{prop:shypPDE_Lexistence_eps}
Suppose that $x \mapsto p(x)r(x)$ is an increasing function, and let $m \in \mathbb{N}$. If $h \in \mathcal{D}_{\mathcal{L}}^{(2)}$ and\, $\ell(h) \in \mathcal{D}_{\mathcal{L}}^{(2)\!}$, then the function
\begin{equation} \label{eq:shypPDE_Lexistence_eps}
f_m(x,y) = \int_{[0,\infty)\!} w_\lambda(x) \, w_{\lambda,m}(y) \, (\mathcal{F} h)(\lambda) \, \bm{\rho}_{\mathcal{L}}(d\lambda)
\end{equation}
is a solution of the Cauchy problem
\begin{equation} \label{eq:shypPDE_Lcauchy_eps}
\begin{aligned}
(\ell_x f_m)(x,y) = (\ell_y f_m)(x,y), & \qquad a < x < b, \: a_m < y < b \\
f_m(x,a_m) = h(x), & \qquad a < x < b \\
(\partial_y^{[1]}\!f_m)(x,a_m) = 0, & \qquad a < x < b.
\end{aligned}
\end{equation}
\end{proposition}

\begin{proof}
Let us begin by justifying that $\partial_x^{[1]}\!f_m(x,y)$ and $(\ell_x f_m)(x,y)$ can be computed via differentiation under the integral sign. The differentiated integrals are given by
\begin{gather}
\label{eq:shypPDE_Lexistence_epssol_pf1} \int_{[0,\infty)\!} w_\lambda^{[1]}(x) \, w_{\lambda,m}(y) \, (\mathcal{F} h)(\lambda) \, \bm{\rho}_{\mathcal{L}}(d\lambda) \\
\label{eq:shypPDE_Lexistence_epssol_pf2} \int_{[0,\infty)\!} w_\lambda(x) \, w_{\lambda,m}(y) \, [\mathcal{F} (\ell (h))](\lambda) \, \bm{\rho}_{\mathcal{L}}(d\lambda)
\end{gather}
(for the latter, we used the identities $(\ell w_\lambda)(x) = \lambda w_\lambda(x)$ and \eqref{eq:shypPDE_Ltransfidentity}), and their absolute and uniform convergence on compacts follows from the fact that $h, \ell (h) \in \mathcal{D}_\mathcal{L}^{(2)}$, together with Lemma \ref{lem:shypPDE_Ltransf_D2prop}(b) and the inequality $|w_{\lambda,m}(\cdot)|\leq 1$ (which follows from Lemma \ref{lem:shypPDE_wsolbound} if we replace $a$ by $a_m$). This justifies that $\partial_x^{[1]}\!f_m(x,y)$ and $(\ell_x f_m)(x,y)$ are given by \eqref{eq:shypPDE_Lexistence_epssol_pf1}, \eqref{eq:shypPDE_Lexistence_epssol_pf2} respectively.

We also need to ensure that $\partial_y^{[1]}\!f_m(x,y)$ and $(\ell_y f_m)(x,y)$ are given by the corresponding differentiated integrals, and to that end we must check that
\[
\int_{[0,\infty)\!} w_\lambda(x) \, w_{\lambda,m}^{[1]}(y) \, (\mathcal{F} h)(\lambda) \, \bm{\rho}_{\mathcal{L}}(d\lambda)
\]
converges absolutely and uniformly. Indeed, it follows from \eqref{eq:shypPDE_ode_wsoleps} that for $y \geq a_m$ we have $w_{\lambda,m}^{[1]}(y) = \lambda \int_{a_m}^y w_{\lambda,m}(\xi) \,r(\xi) d\xi$ and consequently $|w_{\lambda,m}^{[1]}(y)| \leq \lambda \int_{a_m}^y \,r(\xi) d\xi$; hence
\begin{equation} \label{eq:shypPDE_Lexistence_epssol_pf3}
\int_{[0,\infty)\!} \bigl| w_\lambda(x) \,  w_{\lambda,m}^{[1]}(y)\, (\mathcal{F} h)(\lambda)\bigr| \bm{\rho}_{\mathcal{L}}(d\lambda) \leq \int_{a_m}^y r(\xi) d\xi \ccdot\! \int_{[0,\infty)\!} \lambda \bigl| w_\lambda(x) (\mathcal{F} h)(\lambda)\bigr| \bm{\rho}_{\mathcal{L}}(d\lambda)
\end{equation}
and the uniform convergence in compacts follows from \eqref{eq:shypPDE_Ltransfidentity} and Lemma \ref{lem:shypPDE_Ltransf_D2prop}(b).

The verification of the boundary conditions is straightforward: Lemma \ref{lem:shypPDE_Ltransf_D2prop}(b) together with the fact that $w_{\lambda,m}(a_m) = 1$ imply that $f_m(x,a_m) = h(x)$, and from \eqref{eq:shypPDE_Lexistence_epssol_pf3} we easily see that $\partial_y^{[1]}\!f_m (x,a_m) = 0$. This shows that $f_m$ is a solution of the Cauchy problem \eqref{eq:shypPDE_Lcauchy_eps}.
\end{proof}

\begin{corollary} \label{cor:shypPDE_cauchy_ptapprox}
Suppose that $x \mapsto p(x)r(x)$ is an increasing function. Let $h \in \mathcal{D}_\mathcal{L}^{(2)}$ and consider the functions $f_m$, $f$ defined by \eqref{eq:shypPDE_Lexistence_eps}, \eqref{eq:shypPDE_Lexistence}. Then
\[
\lim_{m \to \infty} f_m(x,y) = f(x,y) \qquad \text{pointwise for each } x, y \in (a,b).
\]
\end{corollary}

\begin{proof}
Since $w_{\lambda,m}(y) \to w_\lambda(y)$ pointwise as $m \to \infty$ (Lemma \ref{lem:shypPDE_ode_wepslimit}), the conclusion follows from the dominated convergence theorem (which is applicable due to Lemmas \ref{lem:shypPDE_wsolbound} and \ref{lem:shypPDE_Ltransf_D2prop}(b)).
\end{proof}

\subsection{Maximum principle and positivity of solution}

After having shown that the Cauchy problem is well-posed whenever the function $x \mapsto p(x) r(x)$ is increasing, we introduce a stronger assumption on the coefficients which will be seen to be sufficient for a maximum principle to hold for the hyperbolic equation $\ell_x f = \ell_y f$ and, in consequence, for the solution of the Cauchy problem \eqref{eq:shypPDE_Lcauchy} to preserve positivity and boundedness of its initial condition. We shall rely on the transformation of $\ell$ into the standard form (Remark \ref{rmk:shypPDE_tildeell}); in the following assumption, $A$ is the function defined in \eqref{eq:shypPDE_tildeell_A}.

\renewcommand\theassumption{MP}
\begin{assumption} \label{asmp:shypPDE_SLhyperg}
There exists $\eta \in \mathrm{C}^1(\gamma(a),\infty)$ such that $\eta \geq 0$, $\bm{\phi}_\eta := {A' \over A} - \eta \geq 0$, and the functions $\bm{\phi}_\eta$ and $\bm{\psi}_\eta := {1 \over 2} \eta' - {1 \over 4} \eta^2 + {A' \over 2A} \ccdot \eta$ are both decreasing on $(\gamma(a),\infty)$.
\end{assumption}

Observe that Assumption \ref{asmp:shypPDE_SLhyperg} allows for $\gamma(a) = -\infty$ (this will enable us to treat the case of non-removable degeneracy), and it does not include the left endpoint in the interval where the conditions on $\eta$ are imposed. This assumption is therefore a generalization of an assumption introduced by Zeuner, cf.\ Example \ref{exam:SLhypergr} below.

The proof of the maximum principle presented in the sequel is based on \cite[Proposition 3.7]{zeuner1992} and on the maximum principles of \cite{weinberger1956}. The key tool is the integral identity which we now state:

\begin{lemma}
Let $\bm{\ell}^B$ be the differential expression $\bm{\ell}^B v := - v'' - \bm{\phi}_\eta v' + \bm{\psi}_\eta v$. For $\gamma(a) < c \leq y \leq x$, consider the triangle $\Delta_{c,x,y} := \{(\xi,\zeta) \in \mathbb{R}^2 \mid \zeta \geq c, \, \xi + \zeta \leq x+y, \, \xi - \zeta \geq x-y \}$, and let $v \in \mathrm{C}^2(\Delta_{c,x,y})$. Write $B(x):=\exp({1 \over 2} \int_\beta^x \eta(\xi)d\xi)$ (with $\beta > \gamma(a)$ arbitrary) and $A_{\mathsmaller{B}}(x) = {A(x) \over B(x)^2}$. Then the following integral equation holds:
\begin{equation} \label{eq:shypPDE_inteqtriangle}
A_{\mathsmaller{B}}(x) A_{\mathsmaller{B}}(y) \, v(x,y) = H + I_0 + I_1 + I_2 + I_3 - I_4
\end{equation}
where
\begin{align*}
H & := \tfrac{1}{2} A_{\mathsmaller{B}}(c) \bigl[A_{\mathsmaller{B}}(x-y+c) \, v(x-y+c,c) + A_{\mathsmaller{B}}(x+y-c) \, v(x+y-c,c)] \\
I_0 & := \tfrac{1}{2} A_{\mathsmaller{B}}(c) \int_{x-y+c}^{x+y-c} A_{\mathsmaller{B}}(s) (\partial_y v)(s,c)\, ds \\
I_1 & := \tfrac{1}{2} \int_c^y A_{\mathsmaller{B}}(s) A_{\mathsmaller{B}}(x-y+s) \bigl[ \bm{\phi}_\eta(s) + \bm{\phi}_\eta(x-y+s) \bigr] v(x-y+s,s)\, ds \\
I_2 & := \tfrac{1}{2} \int_c^y A_{\mathsmaller{B}}(s) A_{\mathsmaller{B}}(x+y-s) \bigl[ \bm{\phi}_\eta(s) - \bm{\phi}_\eta(x+y-s) \bigr] v(x+y-s,s)\, ds \\
I_3 & := \tfrac{1}{2} \int_{\Delta_{c,x,y}\!} A_{\mathsmaller{B}}(\xi) A_{\mathsmaller{B}}(\zeta) \bigl[\bm{\psi}_\eta(\zeta) - \bm{\psi}_\eta(\xi)\bigr] v(\xi,\zeta)\, d\xi d\zeta \\
I_4 & := \tfrac{1}{2} \int_{\Delta_{c,x,y}\!} A_{\mathsmaller{B}}(\xi) A_{\mathsmaller{B}}(\zeta) \, (\bm{\ell}_\zeta^B v - \bm{\ell}_\xi^B v)(\xi,\zeta)\, d\xi d\zeta.
\end{align*}
\end{lemma}

\begin{proof}
Just compute 
\begin{align*}
I_4 - I_3 & = \tfrac{1}{2} \int_{\Delta_{c,x,y}}\biggl({\partial \over \partial \xi} \bigl[ A_{\mathsmaller{B}}(\xi) A_{\mathsmaller{B}}(\zeta) \, (\partial_\xi v)(\xi,\zeta) \bigr] - {\partial \over \partial \zeta} \bigl[ A_{\mathsmaller{B}}(\xi) A_{\mathsmaller{B}}(\zeta) \, (\partial_\zeta v)(\xi,\zeta) \bigr] \biggr) d\xi d\zeta \\
& = I_0 - \tfrac{1}{2} \int_0^y A_{\mathsmaller{B}}(s) A_{\mathsmaller{B}}(x-y+s) \, (\partial_\zeta v + \partial_\xi v)(x-y+s,s) \, ds \\
& \quad\; - \tfrac{1}{2} \int_0^y A_{\mathsmaller{B}}(s) A_{\mathsmaller{B}}(x+y-s) \, (\partial_\zeta v - \partial_\xi v)(x+y-s,s) \, ds \\
& = I_0 + I_1 - \int_c^y {d \over ds} \bigl[ A_{\mathsmaller{B}}(s) A_{\mathsmaller{B}}(x-y+s) \, v(x-y+s,s) \bigr] ds \\
& \quad\; + I_2 - \int_c^y {d \over ds} \bigl[ A_{\mathsmaller{B}}(s) A_{\mathsmaller{B}}(x-y+s) \, v(x-y+s,s) \bigr] ds
\end{align*}
where in the second equality we used Green's theorem, and the third equality follows easily from the fact that $(A_{\mathsmaller{B}})' = \bm{\phi}_\eta A_{\mathsmaller{B}}$.
\end{proof}

\begin{theorem}[Weak maximum principle] \label{thm:shypPDE_tildeell_maxprinc}
Suppose Assumption \ref{asmp:shypPDE_SLhyperg} holds, and let $\gamma(a) < c \leq y_0 \leq x_0$. If $u \in \mathrm{C}^2(\Delta_{c,x_0,y_0})$ satisfies
\begin{equation} \label{eq:shypPDE_tildeell_maxprinc_ineq}
\begin{aligned}
(\widetilde{\ell}_x u - \widetilde{\ell}_y u)(x,y) \leq 0, & \qquad (x,y) \in \Delta_{c,x_0,y_0} \\
u(x,c) \geq 0, & \qquad x \in [x_0-y_0+c,x_0+y_0-c] \\
(\partial_y u)(x,c) + \tfrac{1}{2} \eta(c) u(x,c) \geq 0, & \qquad x \in [x_0-y_0+c,x_0+y_0-c]
\end{aligned}
\end{equation}
then $u \geq 0$ in $\Delta_{c,x_0,y_0}$.
\end{theorem}

\begin{proof}
Pick a function $\omega \in \mathrm{C}^2[c,\infty)$ such that $\bm{\ell}^B \omega < 0$, $\omega(c) > 0$ and $\omega'(c) \geq 0$. Clearly, it is enough to show that for all $\delta > 0$ we have $v(x,y) := B(x) B(y) u(x,y) + \delta \omega(y) > 0$ for $(x, y) \in \Delta_{c,x_0,y_0}$.

Assume by contradiction that there exist $\delta > 0$, $(x, y) \in \Delta_{c,x_0,y_0}$ for which we have $v(x,y) = 0$ and $v(\xi,\zeta) \geq 0$ for all $(\xi,\zeta) \in \Delta_{c,x,y} \subset \Delta_{c,x_0,y_0}$. It is clear from the choice of $\omega$ that $v(\cdot,c) > 0$, thus we have $H \geq 0$ in the right hand side of \eqref{eq:shypPDE_inteqtriangle}. Similarly, $(\partial_y v)(\cdot,c) = B(x) B(y) \bigl[(\partial_y u)(\cdot,c) + \tfrac{1}{2} \eta(c) u(\cdot,c)\bigr] + \delta \omega'(c) \geq 0$, hence $I_0 \geq 0$. Since $\bm{\phi}_\eta$ is positive and decreasing and $\bm{\psi}_\eta$ is decreasing (cf.\ Assumption \ref{asmp:shypPDE_SLhyperg}) and we are assuming that $u \geq 0$ on $\Delta_{c,x,y}$, it follows that $I_1 \geq 0$, $I_2 \geq 0$ and $I_3 \geq 0$. In addition, $I_4 < 0$ because $(\bm{\ell}_\zeta^B v - \bm{\ell}_\xi^B v)(\xi,\zeta) = B(x) B(y) (\widetilde{\ell}_\zeta u - \widetilde{\ell}_\xi u)(\xi,\zeta) + (\bm{\ell}^B\omega)(\zeta) < 0$. Consequently, \eqref{eq:shypPDE_inteqtriangle} yields $0 = A_{\mathsmaller{B}}(x) A_{\mathsmaller{B}}(y) v(x,y) \geq -I_4 > 0$. This contradiction shows that $v(x,y) > 0$ for all $(x,y) \in \Delta_{c,x_0,y_0}$.
\end{proof}

Naturally, this weak maximum principle can be restated in terms of the operator $\ell = -{1 \over r} {d \over dx} ( p \, {d \over dx})$; this is left to the reader. As anticipated above, the positivity-preserving property of the Cauchy problem is a by-product of the maximum principle.

\begin{proposition}
Suppose Assumption \ref{asmp:shypPDE_SLhyperg} holds, and let $m \in \mathbb{N}$. If $h \in \mathcal{D}_\mathcal{L}^{(2)}$,\, $\ell(h) \in \mathcal{D}_{\mathcal{L}}^{(2)\!}$ and $h \geq 0$, then the function $f_m$ given by \eqref{eq:shypPDE_Lexistence_eps} is such that 
\begin{equation} \label{eq:shypPDE_positivity_eps}
f_m(x,y) \geq 0 \qquad \text{for } x \geq y > a_m.
\end{equation}
If, in addition, $h \leq C$ (where $C$ is a constant), then $f_m(x,y) \leq C$ for $x \geq y > a_m$.
\end{proposition}

\begin{proof}
It follows from Proposition \ref{prop:shypPDE_Lexistence_eps} that the function $u_m(x,y) := f_m(\gamma^{-1}(x), \gamma^{-1}(y))$ is a solution of the Cauchy problem
\begin{align}
\label{eq:shypPDE_positivity_eps_pf1} (\widetilde{\ell}_x u_m)(x,y) = (\widetilde{\ell}_y u_m)(x,y), & \qquad x, y > \tilde{a}_m \\
\label{eq:shypPDE_positivity_eps_pf2} u_m(x,\tilde{a}_m) = h(\gamma^{-1}(x)), & \qquad x > \tilde{a}_m \\
\label{eq:shypPDE_positivity_eps_pf3} (\partial_y u_m)(x,\tilde{a}_m) = 0, & \qquad x > \tilde{a}_m
\end{align}
where $\tilde{a}_m = \gamma(a_m)$. Clearly, $u_m$ satisfies the inequalities \eqref{eq:shypPDE_tildeell_maxprinc_ineq} for arbitrary $x_0 \geq y_0 \geq \tilde{a}_m$ (here $c = \tilde{a}_m$). By Theorem \ref{thm:shypPDE_tildeell_maxprinc}, $u_m(x_0,y_0) \geq 0$ for all $x_0 \geq y_0 > \tilde{a}_m$; consequently, \eqref{eq:shypPDE_positivity_eps} holds.

The proof of the last statement is straightforward: if we have $h \leq C$, then $\widetilde{u}_m(x,y) = C - u_m(x,y)$ is a solution of \eqref{eq:shypPDE_positivity_eps_pf1} with initial conditions $\widetilde{u}_m(x,\tilde{a}_m) = C - h(\gamma^{-1}(x)) \geq 0$ and \eqref{eq:shypPDE_positivity_eps_pf3}, thus the reasoning of the previous paragraph yields that $C - u_m \geq 0$ for $x \geq y > \tilde{a}_m$.
\end{proof}

The previous result gives the positivity-preservingness for the solution of the nondegenerate Cauchy problem \eqref{eq:shypPDE_Lcauchy_eps}. The extension of this property to the possibly degenerate problem \eqref{eq:shypPDE_Lcauchy} is an immediate consequence of the pointwise convergence result of Corollary \ref{cor:shypPDE_cauchy_ptapprox}:

\begin{corollary} \label{cor:shypPDE_sol_positivity}
Suppose Assumption \ref{asmp:shypPDE_SLhyperg} holds. If $h \in \mathcal{D}_\mathcal{L}^{(2)}$,\, $\ell(h) \in \mathcal{D}_{\mathcal{L}}^{(2)\!}$ and $h \geq 0$, then the function $f$ given by \eqref{eq:shypPDE_Lexistence} is such that 
\[
f(x,y) \geq 0 \qquad \text{for } x, y \in (a,b).
\]
If, in addition, $h \leq C$, then $f(x,y) \leq C$ for $x, y \in (a,b)$.
\end{corollary}

Note that the conclusion holds for all $x,y \in (a,b)$ because the function $f(x,y)$ is symmetric.

\section{Sturm-Liouville translation and convolution} \label{sec:transl_conv}

Assumption \ref{asmp:shypPDE_SLhyperg} will always be in force throughout this and the subsequent sections.

\subsection{Definition and first properties}

In view of the comments made in the Introduction, it is natural to define the $\mathcal{L}$-convolution $\mu * \nu$ ($\mu,\nu \in \mathcal{M}_\mathrm{C}[a,b)$) in order that, for sufficiently well-behaved initial conditions, the integral $\int_{[a,b)} h(\xi) \, (\delta_x * \delta_y)(d\xi)$ coincides with the solution \eqref{eq:shypPDE_Lexistence} of the hyperbolic Cauchy problem. Having this in mind, let us first confirm that the solution of the hyperbolic Cauchy problem can be represented as an integral with respect to a family of positive measures:

\begin{proposition} \label{prop:shypPDE_sol_subprobrep}
Fix $x, y \in [a,b)$. There exists a subprobability measure $\bm{\nu}_{x,y} \in \mathcal{M}_+[a,b)$ such that, for all initial conditions $h \in \mathrm{C}_{\mathrm{c},0}^4$, the solution \eqref{eq:shypPDE_Lexistence} of the hyperbolic Cauchy problem \eqref{eq:shypPDE_Lcauchy} can be written as
\begin{equation} \label{eq:shypPDE_sol_subprobrep}
f_h(x,y) = \int_{[a,b)} h(\xi) \, \bm{\nu}_{x,y}(d\xi) \qquad (h \in \mathrm{C}_{\mathrm{c},0}^4).
\end{equation}
\end{proposition}

\begin{proof}
For each fixed $x, y \in [a,b)$, the right hand side of \eqref{eq:shypPDE_Lexistence} defines a linear functional $\mathrm{C}_{\mathrm{c},0}^4 \ni h \mapsto f_h(x,y) \in \mathbb{C}$. By Corollary \ref{cor:shypPDE_sol_positivity}, $|f_h(x,y)| \leq \|h\|_\infty$ for $h \in \mathrm{C}_{\mathrm{c},0}^4$. Thus it follows from the Hahn-Banach theorem that this functional admits a linear extension $\mathcal{T}_{x,y}: \mathrm{C}_0[a,b) \to \mathbb{C}$ such that $|\mathcal{T}_{x,y} h| \leq \|h\|_\infty$ for all $h \in \mathrm{C}_0[a,b)$. According to the Riesz representation theorem (cf.\ \cite[Theorem 7.3.6]{cohn2013}), $\mathcal{M}_\mathbb{C}[a,b)$ is the dual of $\mathrm{C}_0[a,b)$; we thus have $\mathcal{T}_{x,y} h = \int_{[a,b)} h(\xi) \bm{\nu}_{x,y}(d\xi)$, where $\bm{\nu}_{x,y}$ is a finite complex measure with $\|\bm{\nu}_{x,y}\| \leq 1$. Finally, the fact that $\int_{[a,b)} h(\xi) \bm{\nu}_{x,y}(d\xi) \equiv f_h(x,y) \geq 0$ for all $h \in \mathrm{C}_{\mathrm{c},0}^4$, $h \geq 0$ (Corollary \ref{cor:shypPDE_sol_positivity}) yields that $\bm{\nu}_{x,y} \in \mathcal{M}_+[a,b)$ is a subprobability measure.
\end{proof}

\begin{definition}
Let $\mu, \nu \in \mathcal{M}_{\mathbb{C}}[a,b)$. The measure
\[
(\mu * \nu)(\cdot) = \int_{[a,b)} \int_{[a,b)} \bm{\nu}_{x,y}(\cdot) \, \mu(dx) \, \nu(dy)
\]
is called the \emph{$\mathcal{L}$-convolution} of the measures $\mu$ and $\nu$. The \emph{$\mathcal{L}$-translation} of a function $h \in \mathrm{B}_\mathrm{b}[a,b)$ is defined as
\[
(\mathcal{T}^y h)(x) = \int_{[a,b)} h(\xi) \, \bm{\nu}_{x,y} (d\xi) \equiv \int_{[a,b)} h(\xi) \, (\delta_x * \delta_y)(d\xi), \qquad x,y \in [a,b).
\]
\end{definition}

It follows from this definition, together with \eqref{eq:shypPDE_Lexistence}, that the $\mathcal{L}$-convolution is such that (for $\mu_1,\mu_2,\nu,\pi \in \mathcal{M}_{\mathbb{C}}[a,b)$ and $p_1, p_2 \in \mathbb{C}$):
\begin{enumerate}[itemsep=0pt,topsep=4pt]
\item[\textbf{\emph{(i)}}] $\mu * \nu = \nu * \mu$\, (Commutativity);
\item[\textbf{\emph{(ii)}}] $(\mu * \nu) * \pi = \mu * (\nu * \pi)$\, (Associativity);
\item[\textbf{\emph{(iii)}}] $(p_1 \mu_1 + p_2 \mu_2) * \nu = p_1 (\mu_1 * \nu) + p_2 (\mu_2 * \nu)$\, (Bilinearity);
\item[\textbf{\emph{(iv)}}] $\|\mu * \nu\| \leq \|\mu\| \ccdot \|\nu\|$ (Submultiplicativity);
\item[\textbf{\emph{(v)}}] If $\mu, \nu \in \mathcal{M}_{+}[a,b)$, then $\mu * \nu \in \mathcal{M}_{+}[a,b)$ (Positivity).
\end{enumerate}
Summarizing this, we have:

\begin{proposition} \label{prop:shypPDE_conv_Mbanachalg}
The space $(\mathcal{M}_{\mathbb{C}}[a,b),*)$, equipped with the total variation norm, is a commutative Banach algebra over $\mathbb{C}$ whose identity element is the Dirac measure $\delta_a$.

Moreover, $\mathcal{M}_{+}[a,b)$ is an algebra cone (i.e.\ it is closed under $\mathcal{L}$-convolution, addition and multiplication by positive scalars, and it contains the identity element).
\end{proposition}

\begin{remark}
Given a measure $\mu \in \mathcal{M}_\mathbb{C}[a,b)$, it is natural to define the $\mathcal{L}$-translation by $\mu$ as
\[
(\mathcal{T}^\mu h)(x) := \int_{[a,b)} (\mathcal{T}^y h)(x)\, \mu(dy) \equiv \int_{[a,b)} h(\xi) \, (\delta_x * \mu)(d\xi) \qquad\quad (h \in \mathrm{B}_\mathrm{b}[a,b))
\]
(so that $\mathcal{T}^x \equiv \mathcal{T}^{\delta_x}$ for $a \leq x < b$). It is easy to see that $\|\mathcal{T}^\mu h\|_\infty \leq \| \mu \| \ccdot \|h\|_\infty$ for all $h \in \mathrm{B}_\mathrm{b}[a,b)$ and $\mu \in \mathcal{M}_\mathbb{C}[a,b)$. Observe also that for $h \in \mathrm{C}_{\mathrm{c},0}^4$ we can write (by \eqref{eq:shypPDE_Lexistence} and \eqref{eq:shypPDE_sol_subprobrep})
\begin{equation} \label{eq:shypPDE_gentranslmu_spectrep}
(\mathcal{T}^\mu h)(x) = \int_{[0,\infty)} (\mathcal{F} h)(\lambda) \, w_\lambda(x) \, \widehat{\mu}(\lambda) \, \bm{\rho}_{\mathcal{L}}(d\lambda) \qquad (h \in \mathrm{C}_{\mathrm{c},0}^4)
\end{equation}
or equivalently (cf.\ Proposition \ref{prop:shypPDE_Ltransf})
\begin{equation} \label{eq:shypPDE_gentranslmu_spectrep_F} \bigl(\mathcal{F}(\mathcal{T}^\mu h)\bigr)(\lambda) = \widehat{\mu}(\lambda) (\mathcal{F}h)(\lambda) \qquad (h \in \mathrm{C}_{\mathrm{c},0}^4).
\end{equation}
Due to Lemma \ref{lem:shypPDE_Ltransf_D2prop}, the integral \eqref{eq:shypPDE_gentranslmu_spectrep} converges absolutely and uniformly for $x$ on compact subsets of $(a,b)$.
\end{remark}

\subsection{Sturm-Liouville transform of measures}

An important tool for the subsequent analysis is the extension of the $\mathcal{L}$-transform \eqref{eq:shypPDE_Ltransfdef} to finite complex measures, defined as follows:

\begin{definition}
Let $\mu \in \mathcal{M}_{\mathbb{C}}[a,b)$. The \emph{$\mathcal{L}$-transform} of the measure $\mu$ is the function defined by the integral
\[
\widehat{\mu}(\lambda) = \int_{[a,b)} w_\lambda(x)\, \mu(dx), \qquad \lambda \geq 0.
\]
\end{definition}

The next proposition contains some basic properties of the $\mathcal{L}$-transform of measures which, as one would expect, resemble those of the ordinary Fourier transform (or characteristic function) of finite measures. We recall that, by definition, the complex measures $\mu_n$ converge weakly to $\mu \in \mathcal{M}_\mathbb{C}[a,b)$ if $\lim_n \int_{[a,b)} g(\xi) \mu_n(d\xi) = \int_{[a,b)} g(\xi) \mu(d\xi)$ for all $g \in \mathrm{C}_\mathrm{b}[a,b)$. We also recall that a family $\{\mu_j\} \subset \mathcal{M}_\mathbb{C}[a,b)$ is said to be uniformly bounded if $\sup_j\|\mu_j\| < \infty$, and $\{\mu_j\}$ is said to be tight if for each $\eps > 0$ there exists a compact $K_\eps \subset [a,b)$ such that $\sup_j \,|\mu_j|([a,b) \setminus K_\eps) < \eps$. (These definitions are taken from \cite{bogachev2007}.) In the sequel, the notation $\mu_n \warrow \mu$ denotes weak convergence of measures.

\begin{proposition} \label{prop:shypPDE_Ltransfmeas_props}
The $\mathcal{L}$-transform $\widehat{\mu}$ of $\mu \in \mathcal{M}_\mathbb{C}[a,b)$ has the following properties:
\begin{enumerate}[itemsep=0pt,topsep=4pt]
\item[\textbf{(a)}] $\widehat{\mu}$ is continuous on $[0,\infty)$. Moreover, if a family of measures $\{\mu_j\} \subset \mathcal{M}_\mathbb{C}[a,b)$ is tight and uniformly bounded, then $\{\widehat{\mu_j}\}$ is equicontinuous on $[0,\infty)$.

\item[\textbf{(b)}] Each measure $\mu \in \mathcal{M}_{\mathbb{C}}[a,b)$ is uniquely determined by $\widehat{\mu}$. In particular, each $f \in L_1(r)$ is uniquely determined by $\mathcal{F}f \equiv \widehat{\mu_f}$, where $\mu_f \in \mathcal{M}_{\mathbb{C}}[a,b)$ is defined by $\mu_f(dx) = f(x) r(x) dx$.

\item[\textbf{(c)}]
If $\{\mu_n\}$ is a sequence of measures belonging to $\mathcal{M}_+[a,b)$, $\mu \in \mathcal{M}_+[a,b)$, and $\mu_n \warrow \mu$, then
\[
\widehat{\mu_n} \xrightarrow[\,n \to \infty\,]{} \widehat{\mu} \qquad \text{uniformly for } \lambda \text{ in compact sets.}
\]

\item[\textbf{(d)}] Suppose that $\lim_{x \uparrow b} w_\lambda(x) = 0$ for all $\lambda > 0$. If $\{\mu_n\}$ is a sequence of measures belonging to $\mathcal{M}_+[a,b)$ whose $\mathcal{L}$-transforms are such that
\begin{equation} \label{eq:shypPDE_Ftransf_continuity_hyp}
\widehat{\mu_n}(\lambda) \xrightarrow[\,n \to \infty\,]{} f(\lambda) \qquad \text{pointwise in } \lambda \geq 0
\end{equation}
for some real-valued function $f$ which is continuous at a neighborhood of zero, then $\mu_n \warrow \mu$ for some measure $\mu \in \mathcal{M}_+[a,b)$ such that $\widehat{\mu} \equiv f$.
\end{enumerate}
\end{proposition}

\begin{proof}
\textbf{\emph{(a)}} Let us prove the second statement, which implies the first. Set $C = \sup_j \|\mu_j\|$. Fix $\lambda_0 \geq 0$ and $\eps > 0$. By the tightness assumption, we can choose $\beta \in (a,b)$ such that $|\mu_j|(\beta,b) < \eps$ for all $j$. Since the family $\{w_{(\cdot)}(x)\}_{x \in (a, \beta]}$ is equicontinuous on $[0,\infty)$ (this follows easily from the power series representation of $w_{(\cdot)}(x)$, cf.\ proof of Lemma \ref{lem:shypPDE_ode_wsol}), we can choose $\delta > 0$ such that
\[
|\lambda - \lambda_0| < \delta \quad \implies \quad |w_\lambda(x) - w_{\lambda_0}(x)| < \eps \; \text{ for all } a < x \leq \beta.
\]
Consequently,
\begin{align*}
& \bigl|\widehat{\mu_j}(\lambda) - \widehat{\mu_j}(\lambda_0)\bigr| = \biggl| \int_{(a,b)} \bigl(w_\lambda(x) - w_{\lambda_0}(x)\bigr) \mu_j(dx) \biggr| \\
& \quad \leq\int_{(\beta,b)\!} \bigl|w_\lambda(x) - w_{\lambda_0}(x)\bigr| |\mu_j|(dx) + \int_{(a,\beta]\!} \bigl|w_\lambda(x) - w_{\lambda_0}(x)\bigr| |\mu_j|(dx) \\
& \quad \leq 2\eps + C\eps = (C+2)\eps
\end{align*}
for all $j$, provided that $|\lambda - \lambda_0| < \delta$, which means that $\{\widehat{\mu_j}\}$ is equicontinuous at $\lambda_0$. \\[-8pt]

\textbf{\emph{(b)}} Let $\mu \in \mathcal{M}_\mathbb{C}[a,b)$ be such that $\widehat{\mu}(\lambda) = 0$ for all $\lambda \geq 0$. We need to show that $\mu$ is the zero measure. For each $h \in \mathrm{C}_{\mathrm{c},0}^{4}$, by \eqref{eq:shypPDE_gentranslmu_spectrep} we have
\[
(\mathcal{T}^\mu h)(x) = \int_{[0,\infty)} (\mathcal{F} h)(\lambda) \, w_\lambda(x) \, \widehat{\mu}(\lambda) \, \bm{\rho}_{\mathcal{L}}(d\lambda) \, = \, 0.
\]
Since $h \in \mathrm{C}_{\mathrm{c},0}^{4}$, Theorem \ref{thm:shypPDE_Lexistence} assures that $\lim_{x \downarrow a} (\mathcal{T}^y h)(x) = h(y)$ for $y \geq 0$; therefore, by dominated convergence (which is applicable because $\|\mathcal{T}^y h\|_\infty \leq \|h\|_\infty < \infty$),
\[
0 = \lim_{x \downarrow a} (\mathcal{T}^\mu h)(x) = \lim_{x \downarrow a} \int_{[a,b)\!} (\mathcal{T}^y h)(x)\, \mu(dy) \\
= \int_{[a,b)\!} h(y)\, \mu(dy)
\]
This shows that $\int_{[a,b)\!} h(y)\, \mu(dy) = 0$ for all $h \in \mathrm{C}_{\mathrm{c},0}^{4}$ and, consequently, $\mu$ is the zero measure. \\[-8pt]

\textbf{\emph{(c)}} Since $w_\lambda(\cdot)$ is continuous and bounded, the pointwise convergence $\widehat{\mu_n}(\lambda) \to \widehat{\mu}(\lambda)$ follows from the definition of weak convergence of measures. By Prokhorov's theorem \cite[Theorem 8.6.2]{bogachev2007}, $\{\mu_n\}$ is tight and uniformly bounded, thus (by part (i)) $\{\widehat{\mu_n}\}$ is equicontinuous on $[0,\infty)$. Invoking \cite[Lemma 15.22]{klenke2014}, we conclude that the convergence $\widehat{\mu_n} \to \widehat{\mu}$ is uniform for $\lambda$ in compact sets. \\[-8pt]

\textbf{\emph{(d)}} We only need to show that the sequence $\{\mu_n\}$ is tight and uniformly bounded. Indeed, if $\{\mu_n\}$ is tight and uniformly bounded, then Prokhorov's theorem yields that for any subsequence $\{\mu_{n_k}\}$ there exists a further subsequence $\{\mu_{n_{k_j}}\!\}$ and a measure $\mu \in \mathcal{M}_+[a,b)$ such that $\mu_{n_{k_j}}\!\! \warrow \mu$. Then, due to part (iii) and to \eqref{eq:shypPDE_Ftransf_continuity_hyp}, we have $\widehat{\mu}(\lambda) = f(\lambda)$ for all $\lambda \geq 0$, which implies (by part (ii)) that all such subsequences have the same weak limit; consequently, the sequence $\mu_n$ itself converges weakly to $\mu$.

The uniform boundedness of $\{\mu_n\}$ follows immediately from the fact that $\widehat{\mu_n}(0) = \mu_n[a,b)$ converges. To prove the tightness, take $\eps > 0$. Since $f$ is continuous at a neighborhood of zero, we have ${1 \over \delta} \int_0^{2\delta} \bigl(f(0) - f(\lambda)\bigr)d\lambda \longrightarrow 0$ as $\delta \downarrow 0$; therefore, we can choose $\delta > 0$ such that
\[
\biggl| {1 \over \delta} \int_0^{2\delta} \bigl(f(0) - f(\lambda)\bigr)d\lambda \biggr| < \eps.
\] 
Next we observe that, due to the assumption that $\lim_{x \uparrow b} w_\lambda(x) = 0$ for all $\lambda > 0$, we have $\int_0^{2\delta} \bigl( 1-w_\lambda(x) \bigr) d\lambda \longrightarrow 2\delta$ as $x \uparrow b$, meaning that we can pick $\beta \in (a,b)$ such that
\[
\int_0^{2\delta} \bigl( 1-w_\lambda(x) \bigr) d\lambda \geq \delta \qquad \text{for all } \beta < x < b.
\]
By our choice of $\beta$ and Fubini's theorem,
\begin{align*}
\mu_n\bigl[\beta,b) & = {1 \over \delta} \int_{[\beta,b)} \delta\, \mu_n(dx) \\
& \leq {1 \over \delta} \int_{[\beta,b)} \int_0^{2\delta} \bigl( 1-w_\lambda(x) \bigr) d\lambda\, \mu_n(dx) \\
& \leq {1 \over \delta} \int_{[a,b)} \int_0^{2\delta} \bigl( 1-w_\lambda(x) \bigr) d\lambda\, \mu_n(dx) \\
& = {1 \over \delta} \int_0^{2\delta} \bigl(\widehat{\mu_n}(0) - \widehat{\mu_n}(\lambda)\bigr) d\lambda.
\end{align*}
Hence, using the dominated convergence theorem,
\begin{align*}
\limsup_{n \to \infty} \mu_n[\beta,b) & \leq {1 \over \delta} \limsup_{n \to \infty}\! \int_0^{2\delta} \bigl(\widehat{\mu_n}(0) - \widehat{\mu_n}(\lambda)\bigr) d\lambda \\
& = {1 \over \delta} \int_0^{2\delta}\!\! \lim_{n \to \infty} \bigl(\widehat{\mu_n}(0) - \widehat{\mu_n}(\lambda)\bigr) d\lambda = {1 \over \delta} \int_0^{2\delta} \bigl(f(0) -
f(\lambda)\bigr) d\lambda < \eps
\end{align*}
due to the choice of $\delta$. Since $\eps$ is arbitrary, we conclude that $\{\mu_n\}$ is tight, as desired.
\end{proof}

\begin{remark} \label{rmk:shypPDE_Ltransfmeas_propsrmk}
\textbf{I.} Parts (c) and (d) of the proposition above show that (whenever $\lim_{x \uparrow b} w_\lambda(x) = 0$ for all $\lambda > 0$) the $\mathcal{L}$-transform possesses the following important property: \emph{the $\mathcal{L}$-transform is a topological homeomorphism between $\mathcal{P}[a,b)$ with the weak topology and the set $\widehat{\mathcal{P}}$ of $\mathcal{L}$-transforms of probability measures with the topology of uniform convergence in compact sets.} \\[-8pt]

\noindent \textbf{II.} Recall that, by definition \cite[\S30]{bauer2001}, the measures $\mu_n$ converge vaguely to $\mu$ if $\lim_n \int_{[a,b)} g(\xi) \mu_n(d\xi) = \int_{[a,b)} g(\xi) \mu(d\xi)$ for all $g \in \mathrm{C}_0[a,b)$. Much like weak convergence, vague convergence of measures can be formulated via the $\mathcal{L}$-transform, provided that $\lim_{x \uparrow b} w_\lambda(x) = 0$ for all $\lambda > 0$. Indeed, using $\varrow$ to denote vague convergence of measures, we have:
\begin{enumerate}[itemsep=0pt,leftmargin=1.25cm]
\item[\emph{II.1}] \emph{If $\{\mu_n\} \subset \mathcal{M}_+[a,b)$, $\mu \in \mathcal{M}_+[a,b)$, and $\mu_n \varrow \mu$, then $\lim \widehat{\mu_n}(\lambda) = \widehat{\mu}(\lambda)$ pointwise for each $\lambda > 0$;}

\item[\emph{II.2}] \emph{If $\{\mu_n\} \subset \mathcal{M}_+[a,b)$, $\{\mu_n\}$ is uniformly bounded and $\lim \widehat{\mu_n}(\lambda) = f(\lambda)$ pointwise in $\lambda > 0$ for some function $f \in \mathrm{B}_\mathrm{b}(0,\infty)$, then $\mu_n \varrow \mu$ for some measure $\mu \in \mathcal{M}_+[a,b)$ such that $\widehat{\mu} \equiv f$.}
\end{enumerate}
(The first part is trivial; the second follows from the reasoning in the first paragraph of the proof of (d) in the proposition above, together with the fact that any uniformly bounded sequence of positive measures contains a vaguely convergent subsequence \cite[p.\ 213]{bauer2001}.) \\[-8pt]

\noindent \textbf{III.} Concerning the additional assumption in the above remarks, one can state: \emph{a necessary and sufficient condition for the condition $\lim_{x \uparrow b} w_\lambda(x) = 0$ ($\lambda > 0$) to hold is that $\lim_{x \uparrow b} p(x)r(x) = \infty$.} This fact can be proved using the transformation into the standard form (Remark \ref{rmk:shypPDE_tildeell}) and known results on the asymptotic behavior of solutions of the Sturm-Liouville equation $-u'' - {A' \over A} u' = \lambda u$ (see \cite[proof of Lemma 3.7]{fruchtl2018}).
\end{remark}

\section{The product formula} \label{sec:prodform}

We saw in the previous section that the hyperbolic maximum principle allows us to introduce a convolution measure algebra associated with the Sturm-Liouville operator. The next aims are to develop harmonic analysis on $L_p$ spaces and to study notions such as the continuity of the convolution or the divisibility of measures. However, this requires a fundamental tool, namely the trivialization property $\widehat{\delta_x * \delta_y} = \widehat{\delta_x} \ccdot \widehat{\delta_y}$ for the $\mathcal{L}$-transform or, which is the same, the product formula for its kernel.

\begin{theorem}[Product formula for $w_\lambda$] \label{thm:shypPDE_prodform}
The product $w_\lambda(x) \, w_\lambda(y)$ admits the integral representation
\begin{equation} \label{eq:shypPDE_prodform}
w_\lambda(x) \, w_\lambda(y) = \int_{[a,b)} w_\lambda(\xi)\, (\delta_x * \delta_y)(d\xi), \qquad x, y \in [a,b), \; \lambda \in \mathbb{C}.
\end{equation}
\end{theorem}

Here we present the proof only in the special (nondegenerate) case $\gamma(a) > - \infty$. The proof of the general case is longer and relies on a different regularization argument; the details are given in \cite{sousaetalforth}.

\begin{proof}[Proof of Theorem \ref{thm:shypPDE_prodform} for the case $\gamma(a) > -\infty$]
Assume first that $\ell = -{1 \over A} {d \over dx} (A {d \over dx})$,\, $0 < x < \infty$, and that Assumption \ref{asmp:shypPDE_SLhyperg} holds with $a = \gamma(a) = 0$. Fix $\lambda \in \mathbb{C}$, and let $\{\mb{w}_{\lambda}^{\langle n\rangle \kern-.1em}\}_{n \in \mathbb{N}} \subset \mathrm{C}_{\mathrm{c},0}^4$ be a sequence of functions such that 
\[
\mb{w}_{\lambda}^{\langle n\rangle \kern-.1em}(x) = w_\lambda(x) \quad \text{for } x \in [0,n], \qquad\qquad \mb{w}_{\lambda}^{\langle n\rangle \kern-.1em}(x) = 0 \quad \text{for } x \geq n + 1.
\] 
Let $f^{\langle n\rangle \kern-.1em}(x,y)$ be the unique solution of the hyperbolic Cauchy problem \eqref{eq:shypPDE_Lcauchy} with initial condition $h(x) = \mb{w}_{\lambda}^{\langle n\rangle \kern-.1em}(x)$. Since the family of characteristics for the hyperbolic equation $(\ell_x u)(x,y) = (\ell_y u)(x,y)$ is $x \pm y = \mathrm{const.}$, the solution $f^{\langle n\rangle \kern-.1em}(x,y)$ depends only on the values of the initial condition on the interval $[|x-y|,x+y]$. Observing that the function $\mb{w}_{\lambda}^{\langle n\rangle \kern-.1em}(x) \, \mb{w}_{\lambda}^{\langle n\rangle \kern-.1em}(y)$ is a solution of the hyperbolic equation $(\ell_x u)(x,y) = (\ell_y u)(x,y)$ on the square $(x,y) \in [0,n]^2$, we deduce that
\[
f^{\langle n\rangle \kern-.1em}(x,y) = \mb{w}_{\lambda}^{\langle n\rangle \kern-.1em}(x) \, \mb{w}_{\lambda}^{\langle n\rangle \kern-.1em}(y) = w_\lambda(x) \, w_\lambda(y), \qquad x,y \in [0,\tfrac{n}{2}].
\]
It thus follows from Proposition \ref{prop:shypPDE_sol_subprobrep} that
\[
w_\lambda(x) \, w_\lambda(y) = \int_{[0,\infty)\!} w_\lambda(\xi) \, \bm{\nu}_{x,y}(d\xi), \qquad x,y \in [0,\tfrac{n}{2}]
\]
(note that $\supp(\bm{\nu}_{x,y}) = [|x-y|,x+y]$ because of the domain of dependence of the hyperbolic equation). Since $n$ is arbitrary, the identity holds for all $x,y \in [0,\infty)$, proving that the theorem holds for operators of the form $\ell = -{1 \over A} {d \over dx} (A {d \over dx})$,\, $0 < x < \infty$.

Now, in the general case of an operator $\ell$ of the form \eqref{eq:shypPDE_elldiffexpr}, note that $\gamma(a) > -\infty$ means that $\smash{\sqrt{r(y) \over p(y)}}$ is integrable near $a$, so that we may assume that $\gamma(a) = 0$ (otherwise, replace the interior point $c$ by the endpoint $a$ in the definition of the function $\gamma$). Applying the first part of the proof to the transformed operator $\widetilde{\ell} = -{1 \over A} {d \over d\xi}(A {d \over d\xi})$ defined via \eqref{eq:shypPDE_tildeell_A}, we find that $\widetilde{w}_\lambda(x) \, \widetilde{w}_\lambda(y) = \int_{[0,\infty)\!} \widetilde{w}_\lambda(\xi)\, (\delta_x \kern.1em \widetilde{*} \kern.12em \delta_y)(d\xi)$ for $x, y \in [0,\infty)$, where $\widetilde{w}_\lambda(\xi) := w_\lambda(\gamma^{-1}(\xi))$ and $\widetilde{*}$ is the convolution associated with $\widetilde{\ell}$. We can rewrite this as 
\[
w_\lambda(x) \, w_\lambda(y) = \int_{[a,b)} w_\lambda(\xi) \bigl[\gamma^{-1}(\delta_{\gamma(x)} \kern.1em \widetilde{*} \kern.12em \delta_{\gamma(y)})\bigr] \! (d\xi), \qquad x, y \in [a,b), \; \lambda \in \mathbb{C}
\]
where the measure in the right hand side is the pushforward of the measure $\delta_{\gamma(x)} \kern.1em \widetilde{*} \kern.12em \delta_{\gamma(y)}$ under the map $\xi \mapsto \gamma^{-1}(\xi)$. But one can easily check that the convolutions $*$ and $\widetilde{*}$ are connected by $\delta_x * \delta_y = \gamma^{-1}(\delta_{\gamma(x)} \kern.1em \widetilde{*} \kern.12em \delta_{\gamma(y)})$ (this is a simple consequence of the definition of the convolution and the relation between the operators $\ell$ and $\widetilde{\ell}$), so we are done.
\end{proof}

\begin{corollary} \label{cor:shypPDEconv_basicprops}
Let $\mu, \nu, \pi \in \mathcal{M}_\mathbb{C}[a,b)$.
\begin{enumerate}[itemsep=0pt,topsep=4pt]
\item[\textbf{(a)}] We have $\pi = \mu * \nu$ if and only if
\[
\widehat{\pi}(\lambda) = \widehat{\mu}(\lambda)\, \widehat{\nu}(\lambda) \qquad \text{for all } \lambda \geq 0.
\]

\item[\textbf{(b)}] Probability measures are closed under $\mathcal{L}$-convolution: if $\mu, \nu \in \mathcal{P}[a,b)$, then $\mu * \nu \in \mathcal{P}[a,b)$.
\end{enumerate}
If $\lim_{x \uparrow b} p(x)r(x) = \infty$ holds (cf.\ Remark \ref{rmk:shypPDE_Ltransfmeas_propsrmk}.III), then the following properties also hold:

\begin{enumerate}[itemsep=0pt,topsep=4pt]
\item[\textbf{(c)}]  The mapping $(\mu,\nu) \mapsto \mu * \nu$ is continuous in the weak topology.

\item[\textbf{(d)}] If $h \in \mathrm{C}_\mathrm{b}[a,b)$, then $\mathcal{T}^\mu h \in \mathrm{C}_\mathrm{b}[a,b)$ for all $\mu  \in \mathcal{M}_\mathbb{C}[a,b)$.

\item[\textbf{(e)}] If $h \in \mathrm{C}_0[a,b)$, then $\mathcal{T}^\mu h \in \mathrm{C}_0[a,b)$ for all $\mu \in \mathcal{M}_\mathbb{C}[a,b)$.
\end{enumerate}
\end{corollary}

\begin{proof}
\textbf{\emph{(a)}} Using \eqref{eq:shypPDE_prodform}, we compute
\begin{align*}
\widehat{\mu * \nu}(\lambda) & = \int_{[a,b)\!\!} w_\lambda(x) \, (\mu * \nu)(dx) \\
& = \int_{[a,b)\!} \int_{[a,b)\!} \int_{[a,b)\!} w_\lambda(\xi)\, (\delta_x * \delta_y)(d\xi) \, \mu(dx) \nu(dy)\\
& = \int_{[a,b)\!} \int_{[a,b)\!\!} w_\lambda(x) \, w_\lambda(y) \, \mu(dx) \nu(dy) \: = \: \widehat{\mu}(\lambda) \, \widehat{\nu}(\lambda), \qquad\;\; \lambda \geq 0.
\end{align*}
This proves the ``only if" part, and the converse follows from the uniqueness property in Proposition \ref{prop:shypPDE_Ltransfmeas_props}(b). \\[-8pt]

\textbf{\emph{(b)}} Due to Proposition \ref{prop:shypPDE_conv_Mbanachalg}, it only remains to prove that $(\mu * \nu)[a,b) = 1$\, ($\mu, \nu \in \mathcal{P}[a,b)$). But this follows at once from part (a):
\[
(\mu * \nu)[a,b) = \widehat{\mu * \nu}(0) = \widehat{\mu}(0) \ccdot \widehat{\nu}(0) = \mu[a,b) \ccdot \nu[a,b) = 1.
\]

\textbf{\emph{(c)}} Since $\widehat{\delta_x * \delta_y}(\lambda) = w_\lambda(x) w_\lambda(y)$, Proposition \ref{prop:shypPDE_Ltransfmeas_props}(d) yields that $(x,y) \mapsto \delta_x * \delta_y$ is continuous in the weak topology. Therefore, for $h \in \mathrm{C}_\mathrm{b}[a,b)$ and $\mu_n, \nu_n \in \mathcal{M}_\mathbb{C}[a,b)$ with $\mu_n \warrow \mu$ and $\nu_n \warrow \nu$ we have
\begin{align*}
\lim_n \int_{[a,b)} h(\xi) (\mu_n * \nu_n)(d\xi) & = \lim_n \int_{[a,b)\!} \int_{[a,b)\!} \biggl( \int_{[a,b)\!} h(\xi)\, (\delta_x * \delta_y)(d\xi) \biggr) \mu_n(dx) \nu_n(dy) \\
& = \int_{[a,b)\!} \int_{[a,b)\!} \biggl( \int_{[a,b)\!} h(\xi)\, (\delta_x * \delta_y)(d\xi) \biggr) \mu(dx) \nu(dy) \\
& = \int_{[a,b)} h(\xi) (\mu * \nu)(d\xi)
\end{align*}
due to the continuity of the function in parenthesis. \\[-8pt]

\textbf{\emph{(d)}} Since $(\mathcal{T}^\mu h)(x) = \int_{[a,b)} h(\xi) \, (\delta_x * \mu)(d\xi)$, this follows immediately from part (c) \\[-8pt]

\textbf{\emph{(e)}} It remains to show that $(\mathcal{T}^\mu h)(x) \to 0$ as $x \uparrow b$. Since $w_\lambda(x) \widehat{\mu}(\lambda) \to 0$ as $x \uparrow b$ ($\lambda > 0$), it follows from Remark \ref{rmk:shypPDE_Ltransfmeas_propsrmk}.II that $\delta_x * \mu \varrow \bm{0}$ as $x \uparrow b$, where $\bm{0}$ denotes the zero measure; this means that for each $h \in \mathrm{C}_\mathrm{0}[a,b)$ we have
\[
(\mathcal{T}^\mu h)(x) = \int_{[a,b)\!} h(\xi) (\delta_x * \mu)(d\xi) \longrightarrow \int_{[a,b)\!} h(\xi)\, \bm{0}(d\xi) = 0 \qquad \text{as } x \uparrow b
\]
showing that $\mathcal{T}^\mu h \in \mathrm{C}_0[a,b)$.
\end{proof}

\section{Harmonic analysis on $L_p$ spaces} \label{sec:Lp_harmonic}

For the remainder of this work, the coefficients of $\ell$ will be assumed to satisfy $\lim_{x \uparrow b} p(x)r(x) = \infty$ (cf.\ Remark \ref{rmk:shypPDE_Ltransfmeas_propsrmk}.III), and Assumption \ref{asmp:shypPDE_SLhyperg} continues to be in place.

In this section, we turn our attention to the basic mapping properties of the $\mathcal{L}$-translation and convolution on the Lebesgue spaces $L_p(r)$ ($1 \leq p \leq \infty$). The first result, whose proof depends on the continuity of the mapping $(\mu,\nu) \mapsto \mu * \nu$, ensures that the $\mathcal{L}$-translation defines a linear contraction on $L_p(r)$:

\begin{proposition} \label{prop:shypPDEwl_gentransl_Lpcont}
Let $1 \leq p \leq \infty$ and $\mu \in \mathcal{M}_+[a,b)$. The $\mathcal{L}$-translation $(\mathcal{T}^\mu h)(x) = \int_{[a,b)} h(\xi) \, (\delta_x * \mu)(d\xi),$ is, for each $h \in L_p(r)$, a Borel measurable function of $x$, and we have
\begin{equation} \label{eq:shypPDEwl_gentransl_Lpcont}
\|\mathcal{T}^\mu h\|_{p} \leq \|\mu\| \ccdot \|h\|_{p} \qquad \text{ for all } h \in L_p(r)
\end{equation}
(consequently, $\mathcal{T}^\mu\bigl(L_p(r)\bigr) \subset L_p(r)$).
\end{proposition}

\begin{proof}
It suffices to prove the result for nonnegative $h \in L_p(r)$, $1 \leq p \leq \infty$.

The map $\nu \mapsto \mu * \nu$ is weakly continuous (Corollary \ref{cor:shypPDEconv_basicprops}(c)) and takes $\mathcal{M}_+[a,b)$ into itself. According to \cite[Section 2.3]{jewett1975}, this implies that, for each Borel measurable $h \geq 0$, the function $x \mapsto (\mathcal{T}^\mu h)(x)$ is Borel measurable. It follows that $\int_{[a,b)} g(x) (\mu * r)(dx) := \int_a^b (\mathcal{T}^\mu g)(x) r(x) dx$ ($g \in \mathrm{C}_\mathrm{c}[a,b)$) defines a positive Borel measure. For $a \leq c_1 < c_2 < b$, let $\mathds{1}_{[c_1,c_2)}$ be the indicator function of $[c_1,c_2)$, let $h_n \in \mathrm{C}_{\mathrm{c},0}^4$ be a sequence of nonnegative functions such that $h_n \to \mathds{1}_{[c_1,c_2)}$ pointwise, and write $\mathfrak{C} = \{g \in \mathrm{C}_\mathrm{c}^\infty(a,b) \mid 0 \leq g \leq 1\}$. We compute \vspace{-5.5pt}
\begin{align*}
(\mu * r)[c_1,c_2) & = \lim_n \int_{[a,b)} h_n(x) (\mu * r)(dx) \\
& = \lim_n \sup_{g \in \mathfrak{C}} \int_a^b (\mathcal{T}^\mu h_n)(x) \, g(x) \, r(x) dx \\
& = \lim_n \sup_{g \in \mathfrak{C}} \int_{[0,\infty)} \! (\mathcal{F}h_n)(\lambda) \, (\mathcal{F}g)(\lambda) \, \widehat{\mu}(\lambda) \, \bm{\rho}_{\mathcal{L}}(d\lambda) \\
& = \lim_n \sup_{g \in \mathfrak{C}} \int_a^b h_n(x) \, (\mathcal{T}^\mu g)(x) \, r(x) dx \\
& \leq \|\mu\| \ccdot \lim_n \int_{[a,b)\!} h_n(x) \, r(x) dx \, = \, \|\mu\| \ccdot \int_{[c_1,c_2)\!} r(x)dx \vspace{-2pt}
\end{align*}
where the third and fourth equalities follow from \eqref{eq:shypPDE_gentranslmu_spectrep} and a change of order of integration, and the inequality holds because $\|\mathcal{T}^\mu g\|_\infty \leq \|\mu\| \ccdot \| g \|_\infty \leq \|\mu\|$. Therefore, $\|\mathcal{T}^\mu h\|_1 = \|h\|_{L_1([a,b),\mu * r)} \leq \|\mu\| \ccdot \|h\|_1$ for each Borel measurable $h \geq 0$. Since $\delta_x * \mu \in \mathcal{M}_+[a,b)$, H\"older's inequality yields that $\|\mathcal{T}^\mu h\|_p \leq \|\mu\|^{1/q} \ccdot \|\mathcal{T}^\mu |h|^p\|_1^{1/p} \leq \|\mu\| \ccdot \| h \|_p$ for $1 < p < \infty$.

Finally, if $h \in L_\infty(r)$, $h \geq 0$ then $h = h_\mb{b} + h_\mb{0}$, where $0 \leq h_\mb{b} \leq \|h\|_\infty$ and $h_\mb{0} = 0$ Lebesgue-almost everywhere. Since $\|\mathcal{T}^\mu h_\mb{0}\|_1 \leq \|\mu\| \ccdot \|h_\mb{0}\|_1 = 0$, we have $\mathcal{T}^y h_\mb{0} = 0$ Lebesgue-almost everywhere, and therefore $\|\mathcal{T}^y h\|_\infty = \|\mathcal{T}^y h_\mb{b}\|_\infty \leq \|\mu\| \ccdot \|h\|_\infty$.
\end{proof}

It is natural to define the $\mathcal{L}$-convolution of functions so that the fundamental identity $\mathcal{F}(h*g) = (\mathcal{F}h) \ccdot (\mathcal{F}g)$ holds (where $\mathcal{F}$ denotes the $\mathcal{L}$-transform \eqref{eq:shypPDE_Ltransfdef}):

\begin{definition}
Let $h, g:[a,b) \longrightarrow \mathbb{C}$. If the integral \vspace{-1pt}
\[
(h * g)(x) = \int_a^b (\mathcal{T}^y h)(x)\, g(y)\, r(y) dy = \int_a^b \int_{[a,b)} \! h(\xi) \, (\delta_x * \delta_y)(d\xi)\, g(y)\, r(y) dy \vspace{-1pt}
\]
exists for almost every $x \in [a,b)$, then we call it the \emph{$\mathcal{L}$-convolution} of the functions $h$ and $g$.
\end{definition}

\begin{proposition} \label{prop:shypPDE_conv_Ltransfident}
If $h \in \mathrm{C}_{\mathrm{c},0}^4$ and $g \in L_1(r)$, then $\bigl(\mathcal{F}(h * g)\bigr)(\lambda) = (\mathcal{F}h)(\lambda) (\mathcal{F}g)(\lambda)$ for all $\lambda \geq 0$.
\end{proposition}

\begin{proof}
For $h \in \mathrm{C}_{\mathrm{c},0}^4$ and $g \in L_1(r)$ we have
\begin{align*}
\bigl(\mathcal{F}(h * g)\bigr)(\lambda) & = \int_a^b \int_a^b (\mathcal{T}^\xi h)(x) g(\xi) \, r(\xi) d\xi \, w_\lambda(x) r(x) dx \\
& = \int_a^b \bigl(\mathcal{F}(\mathcal{T}^\xi h)\bigr)(\lambda) \, g(\xi) r(\xi) d\xi \\[3pt]
& \smash{= (\mathcal{F}h)(\lambda) \int_a^b g(\xi) w_\lambda(\xi) r(\xi) d\xi \, = \, (\mathcal{F}h)(\lambda) (\mathcal{F}g)(\lambda)} \\[-14pt]
\end{align*}
where we have used Fubini's theorem and the identity \eqref{eq:shypPDE_gentranslmu_spectrep_F}.
\end{proof}

\begin{proposition}[Young convolution inequality]
Let $p_1,p_2 \in [1, \infty]$ such that ${1 \over p_1} + {1 \over p_2} \geq 1$. For $h \in L_{p_1}(r)$ and $g \in L_{p_2}(r)$, the $\mathcal{L}$-convolution $h * g$ is well-defined and, for $s \in [1, \infty]$ defined by ${1 \over s} = {1 \over p_1} + {1 \over p_2} - 1$, it satisfies
\[
\| h * g \|_s \leq \| h \|_{p_1} \| g \|_{p_2}
\]
(in particular, $h * g \in L_s(r)$). Consequently, the $\mathcal{L}$-convolution is a continuous bilinear operator from $L_{p_1}(r) \times L_{p_2}(r)$ into $L_s(r)$.
\end{proposition}

The proof is given for completeness; it is analogous to that of the Young inequality for the ordinary convolution.

\begin{proof}
Define ${1 \over t_1} = {1 \over p_1} - {1 \over s}$ and ${1 \over t_2} = {1 \over p_2} - {1 \over s}$. Observe that
\[
|(\mathcal{T}^x h)(y)| \, | g(y)| \leq |(\mathcal{T}^x h)(y)|^{p_1/t_1} \, | g(y)|^{p_2/t_2} \bigl[|(\mathcal{T}^x h)(y)|^{p_1} \, | g(y)|^{p_2}\bigr]^{1/s}.
\]
Since ${1 \over s} + {1 \over t_1} + {1 \over t_2} = 1$, we have by H\"older's inequality and \eqref{eq:shypPDEwl_gentransl_Lpcont}
\begin{align*}
& \int_a^b |(\mathcal{T}^x h)(y)| \, | g(y)| r(y) dy \\
& \qquad \leq \biggl( \int_a^b |(\mathcal{T}^x h)(y)|^{p_1} r(y) dy \biggr)^{1/t_1} \biggl( \int_a^b |g(y)|^{p_2} r(y) dy \biggr)^{1/t_2} \biggl( \int_a^b |(\mathcal{T}^x h)(y)|^{p_1} \, | g(y)|^{p_2} r(y)dy \biggr)^{1/s} \\
& \qquad = \|h\|_{p_1}^{p_1/t_1} \|g\|_{p_2}^{p_2/t_2} \biggl( \int_a^b |(\mathcal{T}^x h)(y)|^{p_1} \, | g(y)|^{p_2} r(y)dy \biggr)^{1/s}.
\end{align*}
Using again \eqref{eq:shypPDEwl_gentransl_Lpcont} we conclude that
\[
\| h * g \|_s \leq \|h\|_{p_1}^{p_1/t_1} \, \|g\|_{p_2}^{p_2/t_2} \, \|h\|_{p_1}^{p_1/s} \, \|g\|_{p_2}^{p_2/s} = \|h\|_{p_1} \|g\|_{p_2}. \qedhere
\]
\end{proof}

A consequence of the Young convolution inequality is that the fundamental identity $\bigl(\mathcal{F}(h * g)\bigr)(\lambda) = (\mathcal{F}h)(\lambda) (\mathcal{F}g)(\lambda)$ (Proposition \ref{prop:shypPDE_conv_Ltransfident}) extends, by continuity, to $h \in L_1(r) \cup L_2(r)$ and $g \in L_1(r)$. Another consequence is the Banach algebra property of the space $L_1(r)$:

\begin{corollary}
The Banach space $L_1(r)$, equipped with the convolution multiplication $h \cdot g \equiv h * g$, is a commutative Banach algebra without identity element.
\end{corollary}

\begin{proof}
The Young convolution inequality shows that the $\mathcal{L}$-convolution defines a binary operation on $L_1(r)$ for which the norm is submultiplicative. The commutativity and associativity of the $\mathcal{L}$-convolution are a consequence of the identity $\mathcal{F}(h * g) = (\mathcal{F}h) \ccdot (\mathcal{F}g)$.

Suppose now that there exists $\mathrm{e} \in L_1(r)$ such that $h * \mathrm{e} = h$ for all $h \in L_1(r)$. Then
\[
(\mathcal{F}h)(\lambda) (\mathcal{F}\mathrm{e})(\lambda) = \bigl(\mathcal{F}(h * \mathrm{e})\bigr)(\lambda) = (\mathcal{F}h)(\lambda) \qquad \text{for all } h \in L_1(r) \text{ and } \lambda \geq 0.
\]
Clearly, this implies that $(\mathcal{F}\mathrm{e})(\lambda) = 1$ for all $\lambda \geq 0$. But we know that $\widehat{\delta_a} \equiv 1$, so it follows from Proposition \ref{prop:shypPDE_Ltransfmeas_props}(b) that $\mathrm{e}(x) r(x) dx = \delta_a(dx)$, which is absurd. This shows that the Banach algebra has no identity element.
\end{proof}

\section{Applications to probability theory} \label{sec:probtheory}

\subsection{Infinite divisibility of measures and the Lévy-Khintchine representation}

The set $\mathcal{P} _\mathrm{id}$ of \emph{$\mathcal{L}$-infinitely divisible measures} (or \emph{$\mathcal{L}$-infinitely divisible distributions}) is defined in the obvious way:
\[
\mathcal{P}_\mathrm{id} = \bigl\{ \mu \in \mathcal{P}[a,b) \bigm| \text{for all } n \in \mathbb{N} \text{ there exists } \nu_n \in \mathcal{P}[a,b) \text{ such that } \mu = \nu_n^{*n} \bigr\}
\]
where $\nu_n^{*n}$ denotes the $n$-fold $\mathcal{L}$-convolution of $\nu_n$ with itself.

It is a simple exercise to show that the $\mathcal{L}$-transform of $\mu \in \mathcal{P}_\mathrm{id}$ is of the form
\[
\widehat{\mu}(\lambda) = e^{-\psi_\mu(\lambda)}
\]
where $\psi_\mu$ is continuous, nonnegative and $\psi_\mu(0) = 0$. The function $\psi_\mu$ is called the \emph{$\mathcal{L}$-exponent} of $\mu \in \mathcal{P}_\mathrm{id}$.

As we will see, the exponents of $\mathcal{L}$-infinitely divisible measures admit a representation which is analogous to the well-known Lévy-Khintchine formula for infinitely divisible measures with respect to the ordinary Fourier transform. In the present context, the relevant notions of Poisson and Gaussian measures are defined as follows:

\begin{definition}
Let $\mu \in \mathcal{M}_+[a,b)$. The measure $\mb{e}(\mu) \in \mathcal{P}[a,b)$ defined by
\[
\mb{e}(\mu) = e^{-\|\mu\|} \sum_{k=0}^\infty {\mu^{*k} \over k!}
\]
(the infinite sum converging in the weak topology) is said to be the \emph{$\mathcal{L}$-compound Poisson measure} associated with $\mu$.
\end{definition}

The $\mathcal{L}$-transform of $\mb{e}(\mu)$ can be easily deduced using Corollary \ref{cor:shypPDEconv_basicprops}(a):
\[
\widehat{\mb{e}(\mu)}(\lambda) = e^{-\|\mu\|} \sum_{k=0}^\infty {\widehat{\mu^{*k}}(\lambda) \over k!} = e^{-\|\mu\|} \sum_{k=0}^\infty {\bigl(\widehat{\mu}(\lambda)\bigr)^k \over k!} = \exp\bigl(\widehat{\mu}(\lambda) - \|\mu\|\bigr).
\]
Since $\mb{e}(\mu_1 + \mu_2) = \mb{e}(\mu_1) * \mb{e}(\mu_2)$ ($\mu_1, \mu_2 \in \mathcal{M}_+[a,b)$), every $*$-compound Poisson measure belongs to $\mathcal{P}_\mathrm{id}$.

To motivate the following definition, we observe that it follows from classical results in probability theory (see e.g.\ \cite[Theorem 16.17]{klenke2014} and \cite[\S III.1]{linnikostrovskii1977}) that an infinitely divisible probability measure on $\mathbb{R}^d$ is Gaussian if and only if it has no nontrivial divisors of the form $\mathfrak{e}(\nu)$, where $\nu$ is a finite positive measure on $\mathbb{R}^d$ and $\mathfrak{e}(\nu)$ denotes the (ordinary) compound Poisson measure associated with $\nu$.

\begin{definition}
A measure $\mu \in \mathcal{P}_{\mathrm{id}}$ is called an \emph{$\mathcal{L}$-Gaussian measure} if
\[
\mu = \mb{e}(\nu) * \vartheta \quad \bigl(a > 0 ,\, \nu \in \mathcal{M}_+[a,b),\, \vartheta \in \mathcal{P}_\mathrm{id}\bigr) \qquad \implies \qquad \mb{e}(\nu) = \delta_a.
\]
\end{definition}

We are now ready to state the analogue of the Lévy-Khintchine representation for infinite divisibility with respect to the $\mathcal{L}$-convolution.

\begin{theorem}[Levy-Khintchine type formula] \label{thm:shypPDE_levykhin}
The $\mathcal{L}$-exponent of a measure $\mu \in \mathcal{P}_\mathrm{id}$ can be represented in the form
\begin{equation} \label{eq:shypPDE_levykhin}
\psi_\mu(\lambda) = \psi_\alpha(\lambda) + \int_{(a,b)\!} \bigl( 1 - w_\lambda(x) \bigr) \nu(dx)
\end{equation}
where $\nu$ is a $\sigma$-finite measure on $(a,b)$ which is finite on the complement of any neighbourhood of $a$ and such that
\[
\int_{(a,b)\!} \bigl( 1 - w_\lambda(x) \bigr) \nu(dx) < \infty
\]
and $\alpha$ is an $\mathcal{L}$-Gaussian measure with $\mathcal{L}$-exponent $\psi_\alpha(\lambda)$. Conversely, each function of the form \eqref{eq:shypPDE_levykhin} is an $\mathcal{L}$-exponent of some $\mu \in \mathcal{P}_\mathrm{id}$.
\end{theorem}

\begin{proof}
We only give a sketch of the proof, and refer to \cite{volkovich1988} for details.

Let $\mu \in \mathcal{P}_{\mathrm{id}}$, let $b > a_1 > a_2 > \ldots$ with $\lim a_n = a$, and let $I_n = [a,a_n)$, $J_n = [a_n,b)$. Consider the set $\mathcal{Q}$ of all divisors of $\mu$ of the form $\mb{e}(\pi)$ such that $\pi(I_1) = 0$. One can prove that the set $\mathrm{D}(\mathfrak{P})$ of all divisors (with respect to the $\mathcal{L}$-convolution) of measures $\nu \in \mathfrak{P}$ is relatively compact whenever $\mathfrak{P} \subset \mathcal{P}[a,b)$ is relatively compact (see \cite[Corollary 1]{volkovich1992}); using this fact, it can be shown that $\sup_{\mu = \mb{e}(\pi) \in \mathcal{Q}} \bigl[ \int_{[a,b)} \bigl(1-w_\lambda(x)\bigr) \pi(dx) \bigr] < \infty$ and, consequently, there exists a divisor $\mu_1 = \mb{e}(\pi_1) \in \mathcal{Q}$ such that $\pi_1(J_1)$ is maximal among all elements of $\mathcal{Q}$. Write $\mu = \mu_1 * \alpha_1$ ($\alpha_1 \in \mathcal{P}_\mathrm{id}$). Applying the same reasoning to $\alpha_1$ with $I_1$ replaced by $I_2$, we get $\alpha_1 = \mu_2 * \alpha_2 = \mb{e}(\pi_2) * \alpha_2$. If we perform this successively, we get
\[
\mu = \alpha_n * \beta_n, \qquad \text{where } \, \beta_n = \mu_1 * \mu_2 * \ldots \mu_n, \qquad \mu_k = \mb{e}(\pi_k)
\]
with $\pi_k(I_k) = 0$ and $\pi_k(J_k)$ having the specified maximality property. The sequences $\{\alpha_n\}$ and $\{\beta_n\}$ are relatively compact; letting $\alpha$ and $\beta$ be limit points, we have 
\[
\mu = \alpha * \beta \qquad (\alpha, \beta \in \mathcal{P}_\mathrm{id}).
\]
Suppose, by contradiction, that $\alpha$ is not $\mathcal{L}$-Gaussian, and let $\mb{e}(\eta)$, with $\eta \neq \delta_a$, be a divisor of $\alpha$. Clearly $\eta(J_k) > 0$ for some $k$; given that each $\alpha_n$ divides $\alpha_{n-1}$, we have $\alpha_k = \mb{e}(\eta) * \nu$ ($\nu \in \mathcal{P}_\mathrm{id}$). If we let $\widetilde{\eta}$ be the restriction of $\eta$ to the interval $J_k$, then
\[
\alpha_{k-1} = \mb{e}(\pi_k + \widetilde{\eta}) * \mb{e}(\eta - \widetilde{\eta}) * \nu
\]
which is absurd (because $(\pi_k + \widetilde{\eta})(J_k) > \pi_k(J_k)$, contradicting the maximality property which defines $\pi_k$). To determine the $\mathcal{L}$-exponent of $\beta$, note that $\beta_n = \mb{e}(\varPi_n)$ is the $\mathcal{L}$-compound Poisson measure associated with $\varPi_n := \sum_{k=1}^n \pi_k$, thus $\psi_{\beta_n}(\lambda) = \int_{(a,b)} \bigl(1-w_\lambda(x)\bigr) \varPi_n(dx)$. Since $\{\varPi_n\}$ is an increasing sequence of measures and each $\mb{e}(\varPi_n)$ dividing $\mu$, there exists a $\sigma$-finite measure $\nu$ such that
\[
\psi_{\beta}(\lambda) = \lim_n \int_{(a,b)} \bigl(1-w_\lambda(x)\bigr) \varPi_n(dx) = \lim_n \int_{(a,b)} \bigl(1-w_\lambda(x)\bigr) \nu(dx) < \infty
\]
($\mu \in \mathcal{P}_\mathrm{id}$ ensures the finiteness of the integral); from the relative compactness of $\mathrm{D}(\{\mu\})$ it is possible to conclude that $\nu(J_k) < \infty$ for all $k$.

For the converse, let $\nu_n$ be the restriction of $\nu$ to the interval $J_n$ defined as above. It is verified without difficulty that the right-hand side of \eqref{eq:shypPDE_levykhin} is continuous at zero, hence by Proposition \ref{prop:shypPDE_Ltransfmeas_props}(d) $\alpha * \mb{e}(\nu_n) \warrow \mu \in \mathcal{P}[a,b)$, and $\mu \in \mathcal{P}_\mathrm{id}$ because $\mathcal{P}_\mathrm{id}$ is closed under weak convergence of measures.
\end{proof}

\subsection{Convolution semigroups and their contraction properties}

\begin{definition}
A family $\{\mu_t\}_{t\geq 0} \subset \mathcal{P}[a,b)$ is called an \emph{$\mathcal{L}$-convolution semigroup} if it satisfies the conditions
\begin{itemize}[itemsep=-1pt,topsep=3pt]
\item $\mu_s * \mu_t = \mu_{s+t}$ for all $s, t \geq 0$;
\item $\mu_0 = \delta_a$;
\item $\mu_t \warrow \delta_a$ as $t \downarrow 0$.
\end{itemize}
\end{definition}

A direct consequence of this definition is that
\begin{equation} \label{eq:shypPDE_infdivsemigr_corresp}
\{\mu_t\} \longmapsto \mu_1 \in \mathcal{P}_\mathrm{id}
\end{equation}
defines a one-to-one correspondence holds between the set of $\mathcal{L}$-convolution semigroups and the set of $\mathcal{L}$-infinitely divisible measures. Indeed, if $\{\mu_t\}$ is an $\mathcal{L}$-convolution semigroup, then it is clear that each $\mu_t$ is $\mathcal{L}$-infinitely divisible; and if $\mu \in \mathcal{P}_\mathrm{id}$ has exponent $\psi_\mu(\lambda)$, then $\widehat{\mu_t}(\lambda) = \exp(-t \, \psi_\mu(\lambda))$ defines the unique $\mathcal{L}$-convolution semigroup such that $\mu_1 = \mu$ (the proof of this is analogous to that for the classical convolution, cf.\ \cite[Theorem 29.6]{bauer1996}).

\begin{proposition} \label{prop:shypPDE_fellerLpsemigr}
Let $\{\mu_t\}$ be an $\mathcal{L}$-convolution semigroup. Then
\[
(T_t h)(x) := (\mathcal{T}^{\mu_t}h)(x) = \int_{[a,b)} h(\xi) (\delta_x * \mu_t)(d\xi)
\]
defines a strongly continuous Markovian contraction semigroup $\{T_t\}_{t \geq 0}$ on $\mathrm{C}_0[a,b)$ and on the spaces $L_p(r)$ ($1 \leq p < \infty$), i.e., the following properties hold:
\begin{enumerate}[itemsep=0pt,topsep=4pt]
\item[\textbf{(i)}] $T_t T_s = T_{t+s}$ for all $t, s \geq 0$;
\item[\textbf{(ii)}]$T_t \bigl(\mathrm{C}_0[0,\infty)\bigr) \subset \mathrm{C}_0[0,\infty)$ for all $t \geq 0$;
\item[\textbf{(ii')}] $T_t \bigl(L_p(r)\bigr) \subset L_p(r)$ for all $t \geq 0$ \, ($1 \leq p < \infty$);
\item[\textbf{(iii)}] $T_t \mathds{1} = \mathds{1}$ for all $t \geq 0$, and if $f \in \mathrm{C}_\mathrm{b}[0,\infty)$ satisfies $0 \leq h \leq 1$, then $0 \leq T_t h \leq 1$;
\item[\textbf{(iv)}] $\lim_{t \downarrow 0} \|T_t h - h\|_\infty = 0$ for each $h \in \mathrm{C}_0[0,\infty)$;
\item[\textbf{(iv')}] $\lim_{t \downarrow 0} \|T_t h - h\|_p = 0$ for each $h \in L_p(r)$ \, ($1 \leq p < \infty$).
\end{enumerate}
Moreover, $\{T_t\}$ is translation-invariant: $T_t \mathcal{T}^\nu f = \mathcal{T}^\nu T_t f$ for all $t \geq 0$ and $\nu \in \mathcal{M}_\mathbb{C}[a,b)$.
\end{proposition}

\begin{proof}
Parts (ii), (ii') and (iii) follow at once from Corollary \ref{cor:shypPDEconv_basicprops} and Proposition \ref{prop:shypPDEwl_gentransl_Lpcont}. Concerning part (i) and the translation invariance property, notice that by \eqref{eq:shypPDE_gentranslmu_spectrep_F} we have
\[
\mathcal{F}(\mathcal{T}^\mu (\mathcal{T}^\nu h)) = \widehat{\mu} \cdot \mathcal{F}\mathcal({T}^\nu h) = \widehat{\mu} \cdot \widehat{\nu} \cdot \mathcal{F}h = \widehat{\mu* \nu} \cdot \mathcal{F}h = \mathcal{F}(\mathcal{T}^{\mu*\nu} h) \qquad (h \in \mathrm{C}_{\mathrm{c},0}^4)
\]
so that $\mathcal{T}^\mu (\mathcal{T}^\nu h) = \mathcal{T}^{\mu*\nu} h$ first for $h \in \mathrm{C}_{\mathrm{c},0}^4$ and then, by continuity, for $h \in \mathrm{C}_0[a,b)$ and $h \in L_p(r)$ ($1 \leq p < \infty$).

To prove part (iv) we just need to show that $\lim_{t \downarrow 0} (T_t h)(x) = h(x)$ for all $h \in \mathrm{C}_0[a,b)$ and $x \in [a,b)$, because it is well-known from the theory of Feller semigroups that for a semigroup satisfying (ii) and (iii) this weak continuity property implies the strong continuity of the semigroup (see e.g.\ \cite[Lemma 1.4]{bottcher2013}). But for $h \in \mathrm{C}_0[a,b)$ and $x \in [a,b)$ we clearly have
\[
\lim_{t \downarrow 0} \bigl((T_t h)(x) - h(x)\bigr) = \lim_{t \downarrow 0} \int_{[a,b)\!} \bigl((\mathcal{T}^y h)(x) - h(x)\bigr) \mu_t(dy) = \int_{[a,b)} \bigl((\mathcal{T}^y h)(x) - h(x)\bigr) \delta_a(dy) = 0
\]
showing that (iv) holds.

For part (iv'), let $h \in L_p(r)$, $\eps > 0$ and choose $g \in \mathrm{C}_\mathrm{c}^\infty(a,b)$ such that $\|h - g\|_p \leq \eps$. Then it follows from \eqref{eq:shypPDEwl_gentransl_Lpcont} and part (iv) that
\begin{align*}
\limsup_{t \downarrow 0} \|T_t h - h\|_p & \leq \limsup_{t \downarrow 0} \Bigl(\|T_t h - T_t g\|_p + \|h - g\|_p + \|T_t g - g\|_p\Bigr) \\
& \leq 2\eps + C \ccdot \limsup_{t \downarrow 0} \|T_t g - g\|_\infty \\
& = \, 2\eps
\end{align*}
where $C = [\int_{\supp(g)} r(x) dx]^{1/p}$\, ($C < \infty$ because the support $\supp(g) \subset (a,b)$ is compact). Since $\eps$ is arbitrary, (iv') holds.
\end{proof}

The result for the space $\mathrm{C}_0[a,b)$ means that $\{T_t\}$ is an $\mathcal{L}$-translation-invariant conservative Feller semigroup. This semigroup is also symmetric with to the measure $r(x)dx$, that is, $\int_a^b (T_t h)(x) g(x) r(x) dx = \int_a^b h(x) (T_t g)(x) r(x) dx$ for $h,g \in \mathrm{C}_\mathrm{c}[a,b)$. Any such symmetric Feller semigroup extends to a strongly continuous Markovian contraction semigroup $\{T_t^{(p)}\}_{t \geq 0}$ on $L_p(r)$, $1 \leq p < \infty$ \cite[Lemma 1.45]{bottcher2013}. However, the conclusion of Proposition \ref{prop:shypPDE_fellerLpsemigr} is stronger: it also states that the integral with respect to the Feller transition function is well-defined for all $h \in \cup_{1 \leq p < \infty} L_p(r)$ and, accordingly, the extensions $T_t^{(p)}$ are also given by $h \mapsto (\mathcal{T}^{\mu_t} h)(x) = \int_{[a,b)} h(\xi) (\delta_x * \mu_t)(d\xi)$.

On the Hilbert space $L_2(r)$, we can take advantage of the $\mathcal{L}$-transform to obtain a characterization of the generator of the $L_2$-Markovian semigroup $T_t^{(2)} \equiv T_t: L_2(r) \longrightarrow L_2(r)$:

\begin{proposition} 
Let $\{\mu_t\}$ be an $\mathcal{L}$-convolution semigroup with exponent $\psi$. Then the infinitesimal generator $(\mathcal{A}^{(2)}, \mathcal{D}_{\!\mathcal{A}^{(2)}})$ of the $L_2$-Markovian semigroup $\{T_t^{(2)}\}$ is given by
\[
\mathcal{F}(\mathcal{A}^{(2)} h) = -\psi \ccdot (\mathcal{F}h), \qquad h \in \mathcal{D}_{\!\mathcal{A}^{(2)}}
\]
where
\[
\mathcal{D}_{\!\mathcal{A}^{(2)}} = \biggl\{ h \in L_2(r) \biggm| \int_{[0,\infty)} \bigl|\psi(\lambda)\bigr|^2 \bigl|(\mathcal{F}h)(\lambda)\bigr|^2 \bm{\rho}_{\mathcal{L}}(d\lambda) < \infty \biggr\}.
\]
\end{proposition}

\begin{proof}
We give a proof which follows closely that of the corresponding result for the ordinary convolution, as given in \cite[Theorem 12.16]{bergforst1975}.

Let $h \in \mathcal{D}_{\!\mathcal{A}^{(2)}}$, so that $L_2\text{-\!}\lim_{t \downarrow 0}{1 \over t} (T_t h - h) = \mathcal{A}^{(2)}h \in L_2(\mathrm{m})$. Recalling that (by \eqref{eq:shypPDE_gentranslmu_spectrep_F}) $\mathcal{F}(T_t h) = \widehat{\mu_t} \ccdot (\mathcal{F}h) = e^{-t\psi} \ccdot (\mathcal{F}h)$ for all $h \in L_2(r)$, we see that  
\[
L_2\text{-\!}\lim_{t \downarrow 0}{1 \over t} \bigl(e^{-t\,\psi} - 1\bigr) \ccdot (\mathcal{F}h) = \mathcal{F}(\mathcal{A}^{(2)}h)
\]
The convergence holds almost everywhere along a sequence $\{t_n\}_{n \in \mathbb{N}}$ such that $t_n \to 0$, so we conclude that $\mathcal{F}(\mathcal{A}^{(2)}h) = -\psi \cdot (\mathcal{F}h) \in L_2(\mathbb{R}; \bm{\rho}_\mathcal{L})$.

Conversely, if we let $h \in L_2(r)$ with $-\psi \ccdot (\mathcal{F}h) \in L_2(\mathbb{R}; \bm{\rho}_\mathcal{L})$, then we have
\[
L_2\text{-\!}\lim_{t \downarrow 0} {1 \over t}\bigl(\mathcal{F}({T_t h}) - \mathcal{F}h \bigr) = -\psi \ccdot (\mathcal{F}h) \in L_2(\mathbb{R}; \bm{\rho}_\mathcal{L})
\]
and the isometry gives that $L_2\text{-\!}\lim_{t \downarrow 0} {1 \over t}\bigl(T_t h - h \bigr) \in L_2(\mathrm{m})$, meaning that $h \in \mathcal{D}_{\!\mathcal{A}^{(2)}}$.
\end{proof}

\subsection{Additive and Lévy processes}

\begin{definition}
An $[a,b)$-valued Markov chain $\{S_n\}_{n \in \mathbb{N}_0}$ is said to be \emph{$\mathcal{L}$-additive} if there exist measures $\mu_n \in \mathcal{P}[a,b)$ such that
\begin{equation} \label{eq:shypPDE_markovadditive_def}
P[S_n \in B | S_{n-1} = x] = (\mu_n * \delta_x)(B), \qquad\;\; n \in \mathbb{N}, \: a \leq x < b, \: B \text{ a Borel subset of } [a,b).
\end{equation}
If $\mu_n = \mu$ for all $n$, then $\{S_n\}$ is said to be an \emph{$\mathcal{L}$-random walk}. 
\end{definition}

An explicit construction can be given for $\mathcal{L}$-additive Markov chains, based on the following lemma:

\begin{lemma}
There exists a Borel measurable $\Phi:[a,b) \times [a,b) \times [0,1] \longrightarrow [a,b)$ such that
\[
(\delta_x * \delta_y)(B) = \mathfrak{m}\{\Phi(x,y,\cdot) \in B\}, \qquad x, y \in [a,b), \; B \text{ a Borel subset of } [a,b)
\]
where $\mathfrak{m}$ denotes Lebesgue measure on $[0,1]$.
\end{lemma}

\begin{proof}
Let $\Phi(x,y,\xi) = \max\bigl( a, \sup\{ z \in [a,b): (\delta_x * \delta_y)[a,z] < \xi \} \bigr)$. Using the continuity of the $\mathcal{L}$-convolution, one can show that $\Phi$ is Borel measurable, see \cite[Theorem 7.1.3]{bloomheyer1994}. It is straightforward that $\mathfrak{m}\{\Phi(x,y,\cdot) \in [a,c]\} = \mathfrak{m}\{ (\delta_x * \delta_y)[a,c] \geq \xi \} = (\delta_x * \delta_y)[a,c]$.
\end{proof}

Let $X_1$, $U_1$, $X_2$, $U_2$, $\ldots$ be a sequence of independent random variables (on a given probability space $(\Omega, \mathfrak{A}, \bm{\pi})$) where the $X_n$ have distribution $P_{X_n} = \mu_n \in \mathcal{P}[a,b)$ and each of the (auxiliary) random variables $U_n$ has the uniform distribution on $[0,1]$. Set 
\begin{equation} \label{eq:shypPDE_markovadditive_constr}
S_0 = 0, \qquad S_n = S_{n-1} \oplus_{U_n} X_n
\end{equation}
where $X \oplus_U Y := \Phi(X,Y,U)$. Then we have $P_{S_n} = P_{S_{n-1}} * \mu_n$ ($n \in \mathbb{N}_0$) and, consequently, $\{S_n\}_{n \in \mathbb{N}_0}$ is an $\mathcal{L}$-additive Markov chain satisfying \eqref{eq:shypPDE_markovadditive_def}. The identity $P_{S_n} = P_{S_{n-1}} * \mu_n$ is easily checked:
\begin{align*}
P_{S_n}(B) = P\bigl[ \Phi(S_{n-1},X_n,U_n) \in B \bigr] & = \int_{[a,b)} \! \int_{[a,b)} \mathfrak{m}\{\Phi(x,y,\cdot) \in B\} P_{S_{n-1}}(dx) P_{X_n}(dy) \\
& = \int_{[a,b)} \! \int_{[a,b)} (\delta_x * \delta_y)(B) P_{S_{n-1}}(dx) P_{X_n}(dy) \\
& = (P_{S_{n-1}} * \mu_n)(B).
\end{align*}

We now define the continuous-time analogue of $\mathcal{L}$-random walks:

\begin{definition}
An $[a,b)$-valued Markov process $Y = \{Y_t\}_{t \geq 0}$ is said to be an \emph{$\mathcal{L}$-Lévy process} if there exists an $\mathcal{L}$-convolution semigroup $\{\mu_t\}_{t\geq 0}$ such that the transition probabilities of $Y$ are given by
\[
P\bigl[Y_t \in B | Y_s = x\bigr] = (\mu_{t-s} * \delta_x)(B), \qquad 0 \leq s \leq t,\; a \leq x < b,\; B \text{ a Borel subset of } [a,b).
\]
\end{definition}

The notion of an $\mathcal{L}$-Lévy process coincides with that of a Feller process associated with the Feller semigroup $T_t f = \mathcal{T}^{\mu_{t\!}} f$. Consequently, the general connection between Feller semigroups and Feller processes (see e.g.\ \cite[Section 1.2]{bottcher2013}) ensures that for each (initial) distribution $\nu \in \mathcal{P}[a,b)$ and $\mathcal{L}$-convolution semigroup $\{\mu_t\}_{t\geq 0}$ there exists an $\mathcal{L}$-Lévy process $Y$ associated with $\{\mu_t\}_{t \geq 0}$ and such that $P_{Y_0} = \nu$. Any $\mathcal{L}$-Lévy process has the following properties:
\begin{itemize}[itemsep=-1pt,topsep=3pt] \item It is stochastically continuous: $Y_s \to Y_t$ in probability as $s \to t$, for each $t \geq 0$;
\item It has a càdlàg modification: there exists an $\mathcal{L}$-Lévy process $\{\widetilde{Y}_t\}$ with a.s.\ right-continuous paths and satisfying $P\bigl[Y_t = \widetilde{Y}_t \bigr] = 1$ for all $t \geq 0$.
\end{itemize}
(These properties hold for all Feller processes, cf.\ \cite[Section 1.2]{bottcher2013}.)

An analogue of the well-known theorem on appoximation of Lévy processes by triangular arrays holds for $\mathcal{L}$-Lévy processes (below the notation $\darrow$ stands for convergence in distribution):

\begin{proposition}
Let $X$ be an $[a,b)$-valued random variable. The following assertions are equivalent:
\begin{enumerate}[itemsep=0pt,topsep=4pt]
\item[\textbf{(i)}] $X = Y_1$ for some $\mathcal{L}$-Lévy process $Y = \{Y_t\}_{t \geq 0}$.

\item[\textbf{(ii)}] The distribution of $X$ is $\mathcal{L}$-infinitely divisible;

\item[\textbf{(iii)}] $S_{m_n}^n \darrow X$ for some sequence of $\mathcal{L}$-random walks $S^1, S^2,\, \ldots$ (with $S_0^j = a$) and some integers $m_n \to \infty$.
\end{enumerate}
\end{proposition}

\begin{proof}
The equivalence between (i) and (ii) is a restatement of the one-to-one correspondence \eqref{eq:shypPDE_infdivsemigr_corresp} between $\mathcal{L}$-infinitely divisible measures and $\mathcal{L}$-convolution semigroups. It is obvious that (i) implies (iii): simply let $m_n = n$ and $S^n$ the random walk whose step distribution is the law of $Y_{1/n}$.

Suppose that (iii) holds and let $\pi_n$, $\mu$ be the distributions of $S_j^n$, $X$ respectively. Choose $\eps > 0$ small enough so that $\widehat{\mu}(\lambda) > C_\eps > 0$ for $\lambda \in [0,\eps]$, where $C_\eps > 0$ is a constant. By (iii) and Proposition \ref{prop:shypPDE_Ltransfmeas_props}(c), $\widehat{\pi_n}(\lambda)^{m_n} \to \widehat{\mu}(\lambda)$ uniformly on compacts, which implies that $\widehat{\pi_n}(\lambda) \to 1$ for all $\lambda \in [0,\eps]$ and, therefore, by Proposition \ref{prop:shypPDE_Ltransfmeas_props}(d) $\pi_n \warrow \delta_a$. Now let $k \in \mathbb{N}$ be arbitrary. Since $\pi_n \warrow \delta_a$, we can assume that each $m_n$ is a multiple of $k$. Write $\nu_n = \pi_n^{*(m_n/k)}$, so that $\nu_n^{*k} \warrow \mu$. By relative compactness of  $\mathrm{D}(\{\pi_n^{*m_n}\})$ (see the proof of Theorem \ref{thm:shypPDE_levykhin}), the sequence $\{\nu_n\}_{n \in \mathbb{N}}$ has a weakly convergent subsequence, say $\nu_{n_j} \warrow \mu_k$ as $j \to \infty$, and from this it clearly follows that $\mu_k^{*k} = \mu$. Consequently, (ii) holds.
\end{proof}

As one would expect, the diffusion process generated by the Sturm-Liouville operator \eqref{eq:shypPDE_elldiffexpr} (cf.\ Lemma \ref{lem:shypPDE_Lb_fellergen}) is an $\mathcal{L}$-Lévy process:

\begin{proposition}
The irreducible diffusion process $X$ generated by $(\mathcal{L}^{(\mathrm{b})}, \mathcal{D}_\mathcal{L}^{(\mathrm{b})})$ is an $\mathcal{L}$-Lévy process.
\end{proposition}

\begin{proof}
For $t \geq 0$, $a \leq x < b$ let us write $p_{t,x}(dy) \equiv P_x[X_t \in dy]$. Recall from Lemma \ref{lem:shypPDE_Lb_diffusion_tpdf} that
\[
p_{t,x}(dy) \equiv p(t,x,y) r(y) dy = \int_{[0,\infty)} e^{-t\lambda} w_\lambda(x) \, w_\lambda(y) \, \bm{\rho}_\mathcal{L}(d\lambda) \, r(y) dy, \qquad t > 0,\: \, a < x < b
\]
where the integral converges absolutely. Consequently, by Proposition \ref{prop:shypPDE_Ltransf},
\[
\widehat{p_{t,x}}(\lambda) = e^{-t\lambda} w_\lambda
(x), \qquad t \geq 0,\: a \leq x < b
\]
(the weak continuity of $p_{t,x}$ justifies that the equality also holds for $t = 0$ and for $x = a$). This shows that $p_{t,x} = p_{t,a} * \delta_x$ where $\widehat{p_{t,a}}(\lambda) = e^{-t\lambda}$. It is clear from the properties of the $\mathcal{L}$-transform that $\{p_{t,a}\}_{t \geq 0}$ is an $\mathcal{L}$-convolution semigroup; therefore, $X$ is an $\mathcal{L}$-Lévy process.
\end{proof}

An $\mathcal{L}$-convolution semigroup $\{\mu_t\}_{t \geq 0}$ such that $\mu_1$ is an $\mathcal{L}$-Gaussian measure is said to be an \emph{$\mathcal{L}$-Gaussian convolution semigroup}, and an $\mathcal{L}$-Lévy process associated with an $\mathcal{L}$-Gaussian convolution semigroup is called an \emph{$\mathcal{L}$-Gaussian process}.

It actually turns out that the diffusion $X$ generated by $(\mathcal{L}^{(\mathrm{b})}, \mathcal{D}_\mathcal{L}^{(\mathrm{b})})$ is an $\mathcal{L}$-Gaussian process. This is a consequence of the following characterization of $\mathcal{L}$-Gaussian measures:

\begin{proposition}
Let $Y = \{Y_t\}_{t \geq 0}$ be an $\mathcal{L}$-Lévy process, let $\{\mu_t\}_{t \geq 0}$ be the associated $\mathcal{L}$-convolution semigroup and let $(\mathcal{G},\mathcal{D}(\mathcal{G}))$ be the $\mathrm{C}_\mathrm{b}$-generator of the process $Y$. The following conditions are equivalent:
\begin{enumerate}[itemsep=0pt,topsep=4pt]
\item[\textbf{(i)}] $\mu_1$ is a Gaussian measure;
\item[\textbf{(ii)}] $\lim_{t \downarrow 0} {1 \over t} \mu_t\bigl([a,b) \setminus \mathcal{V}_a\bigr) = 0$ for every neighbourhood $\mathcal{V}_a$ of the point $a$;
\item[\textbf{(iii)}] $\lim_{t \downarrow 0} {1 \over t} (\mu_t*\delta_x)\bigl([a,b) \setminus \mathcal{V}_x\bigr) = 0 \;$ for every $x \in [a,b)$ and every neighbourhood $\mathcal{V}_x$ of the point $x$;
\item[\textbf{(iv)}] $Y$ has a modification whose paths are a.s.\ continuous.
\end{enumerate}
If any of these conditions hold then the $\mathrm{C}_\mathrm{b}$-generator of $Y$ is a local operator, i.e., $(\mathcal{G}h)(x) = (\mathcal{G}g)(x)$ whenever $h, g \in \mathcal{D}(\mathcal{G})$ and $h=g$ on some neighbourhood of $x \in [a,b)$.
\end{proposition}

\begin{proof}
\textbf{\emph{(i)$\!\implies\!$(ii):\;}} Let $\{t_n\}_{n \in \mathbb{N}}$ be a sequence such that $t_n \to 0$ as $n \to \infty$, and let $\nu_n = \mb{e}\bigl(\tfrac{1}{t_n} \mu_{t_{n\!}}\bigr)$. We have 
\begin{equation} \label{eq:shypPDE_gauss_equivdef_pf1}
\lim_{n \to \infty} \widehat{\nu_n}(\lambda) = \lim_{n \to \infty} \exp\biggl[ {1 \over t_n}\bigl(\widehat{\mu_{1}}(\lambda)^{t_n} - 1\bigr) \biggr] = \widehat{\mu_1}(\lambda), \qquad \lambda > 0
\end{equation}
and therefore, by  Proposition \ref{prop:shypPDE_Ltransfmeas_props}(d), $\nu_n \warrow \mu_1$ as $n \to \infty$. From this it follows, cf.\ \cite{volkovich1988}, that if $\pi_n$ denotes the restriction of ${1 \over t_n} \mu_{t_n}$ to $[a,b) \setminus \mathcal{V}_a$, then $\{\pi_n\}$ is relatively compact; if $\pi$ is a limit point, then $\mb{e}(\pi)$ is a divisor of $\mu_1$. Since $\mu_1$ is Gaussian, $\mb{e}(\pi) = \delta_a$, hence $\pi$ must be the zero measure, showing that (ii) holds. \\[-10pt]

\textbf{\emph{(ii)$\!\implies\!$(i):\;}} As in \eqref{eq:shypPDE_gauss_equivdef_pf1},
\[
\widehat{\mu_1}(\lambda) = \lim_{n \to \infty} \exp\biggl[ {1 \over t_n} \int_{[a,b)} \bigl(w_\lambda(x) - 1\bigr) \mu_{t_n\!}(dx) \biggr] = \lim_{n \to \infty} \exp\biggl[ {1 \over t_n} \int_{\mathcal{V}_a} \bigl(w_\lambda(x) - 1\bigr) \mu_{t_n\!}(dx) \biggr], \qquad \lambda > 0
\]
where the second equality is due to (ii), noting that ${1 \over t_n} \int_{[a,b) \setminus \mathcal{V}_a} (w_\lambda(x) - 1) \mu_{t_n\!}(dx) \leq {2 \over t} \mu_{t_n}\bigl([a,b) \setminus \mathcal{V}_a\bigr)$. Given that $\nu_n = \mb{e}\bigl(\tfrac{1}{t_n} \mu_{t_{n\!}}\bigr) \warrow \mu_1$, we have (again, see \cite{volkovich1988})
\[
\widehat{\mu}_1(\lambda) = \exp\biggl[ \int_{(a,b)} \bigl(w_\lambda(x) - 1\bigr) \, \eta(dx) \biggr], \qquad \lambda > 0
\]
for some $\sigma$-finite measure $\eta$ on $(a,b)$ which, by the above, vanishes on the complement of any neighbourhood of the point $a$. Therefore, $\mu_1$ is Gaussian. \\[-10pt]

\textbf{\emph{(ii)$\!\iff\!$(iii):\;}} To prove the nontrivial direction, assume that (ii) holds, and fix $x \in (a,b)$. Let $\mathcal{V}_x$ be a neighbourhood of the point $x$ and write $E_x = [a,b) \setminus \mathcal{V}_x$. Pick a function $h \in \mathrm{C}_{\mathrm{c},0}^4$ such that $0 \leq h \leq 1$, $h = 0$ on $E_x$ and $h = 1$ on some smaller neighbourhood  $\mathcal{U}_x \subset \mathcal{V}_x$ of the point $x$.

We begin by showing that
\begin{equation} \label{eq:shypPDE_gauss_equivdef_pf2}
\lim_{y \downarrow a} {1 - (\mathcal{T}^x h)(y) \over 1-w_\lambda(y)} = 0 \qquad \text{for each } \lambda > 0.
\end{equation}
Indeed, it follows from Theorem \ref{thm:shypPDE_Lexistence} that $\lim_{y \downarrow a} (\mathcal{T}^x h)(y) = 1$, $\lim_{y \downarrow a} \partial_y^{[1]} (\mathcal{T}^x h)(y) = 0$ and
\[
\ell_y (\mathcal{T}^x h)(y) = \int_{[0,\infty)\!} \lambda\, (\mathcal{F} h)(\lambda) \, w_\lambda(x) \, w_\lambda(y) \, \bm{\rho}_{\mathcal{L}}(d\lambda) = \bigl(\mathcal{T}^x \ell(h)\bigr)(y) \xrightarrow[\,y \downarrow a\,]{} \ell(h)(x) = 0,
\]
hence using L'Hôpital's rule twice we find that $\lim_{y \downarrow a} {1 - (\mathcal{T}^x h)(y) \over 1-w_\lambda(y)} = \lim_{y \downarrow a} {\ell_y(\mathcal{T}^x h)(y) \over \lambda w_\lambda(y)} = 0$ ($\lambda > 0$).

By \eqref{eq:shypPDE_gauss_equivdef_pf2}, for each $\lambda > 0$ there exists $a_\lambda > a$ such that $(\mathcal{T}^x \mathds{1}_{E_x})(y) \leq \bigl(\mathcal{T}^x (\mathds{1} - h)\bigr)(y) \leq 1-w_\lambda(x)$ for all $y \in [a,a_\lambda)$ (here $\mathds{1}_{E_x}$ denotes the indicator function of $E_x$). We then estimate
\begin{align*}
{1 \over t} (\mu_t*\delta_x)(E_x) & = {1 \over t} \int_{[a,b)\!} (\mathcal{T}^x \mathds{1}_{E_x})(y) \mu_t(dy) \\
& \leq {1 \over t} \int_{[a,a_\lambda)\!} \bigl( 1 - w_\lambda(y) \bigr) \mu_t(dy) + {1 \over t} \mu_t [a_\lambda,b) \\
& \leq {1 \over t} \int_{[a,b)\!} \bigl( 1 - w_\lambda(y) \bigr) \mu_t(dy) + {1 \over t} \mu_t [a_\lambda,b) \\
& = {1 \over t} \bigl( 1-\widehat{\mu_t}(\lambda) \bigr) + {1 \over t} \mu_t [a_\lambda,b).
\end{align*}
Given that we are assuming that (ii) holds and, by the $\mathcal{L}$-semigroup property, $\lim_{t \downarrow 0} {1 \over t} \bigl( 1-\widehat{\mu_t}(\lambda) \bigr) = \lim_{t \downarrow 0} {1 \over t} \bigl( 1-\widehat{\mu_1}(\lambda)^t \bigr) = -\log\widehat{\mu_1}(\lambda)$, the above inequality gives
\[
\limsup_{t \downarrow 0} {1 \over t} (\mu_t*\delta_x)(E_x) \leq -\log\widehat{\mu_1}(\lambda).
\]
This holds for arbitrary $\lambda > 0$. Since the right-hand side is continuous and vanishes for $\lambda = 0$, we conclude that $\lim_{t \downarrow 0} {1 \over t} (\mu_t*\delta_x)(E_x) = 0$, as desired. \\[-10pt]

\textbf{\emph{(iii)$\!\implies\!$(iv):\;}} This follows from a general result in the theory of Feller processes \cite[Chapter 4, Proposition 2.9]{ethierkurtz1986} according to which $\lim_{t \downarrow 0} {1 \over t} P_x[Y_t \in [a,b) \setminus \mathcal{V}_x] = 0$ is a sufficient condition for a given $[a,b)$-valued Feller process $Y$ to have continuous paths. \\[-10pt]

\textbf{\emph{(iv)$\!\implies\!$(iii):\;}} This is a consequence of Ray's theorem for one-dimensional Markov processes, which is stated and proved in \cite[Theorem 5.2.1]{ito2006}. \\[-10pt]

Finally, it is well-known that Markov processes with continuous paths have local generators (see e.g.\ \cite[Theorem 5.1.1]{ito2006}), thus the last assertion holds.
\end{proof}

To finish this section, it is worth mentioning that analogues of the classical limit theorems --- such as laws of large numbers or central limit theorems --- can be established for the $\mathcal{L}$-convolution measure algebra. As in the setting of hypergroup convolution structures (cf.\ Example \ref{exam:SLhypergr}), solutions $\{\varphi_k\}_{k \in \mathbb{N}}$ of the functional equation
\[
(\mathcal{T}^y \varphi_k)(x) = \sum_{j=0}^k \binom{k}{j} \varphi_j(x) \varphi_{k-j}(y) \quad\; \bigl(x,y \in [a,b)\bigr), \qquad\; \varphi_0 = 0,
\]
which are called \emph{$\mathcal{L}$-moment functions}, play a role similar to that of the monomials under the ordinary convolution. 

For the sake of illustration, let us state some strong laws of large numbers which hold true for the $\mathcal{L}$-convolution: let $\{S_n\}$ be an $\mathcal{L}$-additive Markov chain constructed as in \eqref{eq:shypPDE_markovadditive_constr}, and define the $\mathcal{L}$-moment functions of first and second order by $\varphi_1(x) = \kappa \eta_1(x)$, $\varphi_2(x) = 2[\kappa \eta_2(x) + \eta_1(x)]$ respectively, where $\kappa := \lim_{\xi \to \infty} {A'(\xi) \over A(\xi)} = \lim_{x \uparrow b} {[(pr)^{1/2}]'\!(x) \over 2r(x)}$ and the $\eta_j$ are given by \eqref{eq:shypPDE_wsol_powseries_eta}. Then: \\[-8pt]
\begin{enumerate}[itemsep=0pt,topsep=0pt,leftmargin=1.75cm]
\item[\textbf{7.13.I.}] \emph{If $\{r_n\}_{n \in \mathbb{N}}$ is a sequence of positive numbers such that $\lim_n r_n = \infty$ and $\sum_{n=1}^\infty {1 \over r_n} \bigl(\mathbb{E}[\varphi_2(X_n)] - \mathbb{E}[\varphi_1(X_n)]^2\bigr) < \infty$, then
\[
\lim_n {1 \over \sqrt{r_n}} \bigl( \varphi_1(S_n) - \mathbb{E}[\varphi_1(S_n)] \bigr) = 0 \qquad\;\; \bm{\pi}\text{-a.s.}
\]}
\item[\textbf{7.13.II.}] \emph{If $\{S_n\}$ is an $\mathcal{L}$-random walk such that $\mathbb{E}[\varphi_2(X_1)^{\theta/2}] < \infty$ for some $1 \leq \theta < 2$, then $\mathbb{E}[\varphi_1(X_1)] < \infty$ and
\[
\lim_n {1 \over n^{1 / \theta}} \bigl(\varphi_1(S_n) - n\mathbb{E}[\varphi_1(X_1)]\bigr) = 0 \qquad\;\; \bm{\pi}\text{-a.s.} 
\]}
\item[\textbf{7.13.III.}] \emph{Suppose that $\varphi_1 \equiv 0$. If $\{r_n\}_{n \in \mathbb{N}}$ is a sequence of positive numbers such that $\lim_n r_n = \infty$ and $\sum_{n=1}^\infty {1 \over r_n} \mathbb{E}[\varphi_2(X_n)] < \infty$, then
\[
\lim_n {1 \over r_n} \varphi_2(S_n) = 0 \qquad\;\; \bm{\pi}\text{-a.s.}
\]}
\item[\textbf{7.13.IV.}] \emph{Suppose that $\varphi_1 \equiv 0$. If $\{S_n\}$ is an $\mathcal{L}$-random walk such that $\mathbb{E}[\varphi_2(X_1)^{\theta}] < \infty$ for some $0 < \theta < 1$, then
\[
\lim_n {1 \over n^{1 / \theta}} \varphi_2(S_n) = 0 \qquad\;\; \bm{\pi}\text{-a.s.}
\]}
\end{enumerate}
The above assertions can be proved exactly as in the hypergroup framework: the reader is referred to \cite[Section 7]{zeuner1992}.

\section{Examples} \label{sec:examples}

We begin with two simple examples where the Sturm-Liouville operator is regular and nondegenerate, and the kernel of the $\mathcal{L}$-transform can be written in terms of elementary functions.

\begin{example}[Cosine Fourier transform]
Consider the Sturm-Liouville operator
\[
\ell = - {d^2 \over dx^2}, \qquad 0 < x < \infty
\]
which is obtained by setting $p = r = \mathds{1}$ and $(a,b) = (0,\infty)$. This operator trivially satisfies assumption \ref{asmp:shypPDE_SLhyperg}. Since the solution of the Sturm-Liouville boundary value problem \eqref{eq:shypPDE_ode_wsol} is $w_\lambda(x) = \cos(\tau x)$ (where $\lambda = \tau^2$), the $\mathcal{L}$-transform is simply the cosine Fourier transform $(\mathcal{F} h)(\tau) = \int_0^\infty h(x) \cos(\tau x) dx$. By elementary trigonometric identities, $w_\tau(x) w_\tau(y) = {1 \over 2} [w_\tau(|x-y|) + w_\tau(x+y)]$, hence the $\mathcal{L}$-convolution is given by
\[
\delta_x * \delta_y = {1 \over 2}(\delta_{|x-y|} + \delta_{x+y}), \qquad x,y \geq 0.
\]
In other words, $*$ is (up to identification) the ordinary convolution of symmetric measures.
\end{example}

\begin{example}
If we let $p(x) = r(x) = (1+x)^2$ and $(a,b) = (0,\infty)$, we obtain the differential operator
\[
\ell = - {d^2 \over dx^2} - {2 \over 1+x} {d \over dx}, \qquad 0 < x < \infty,
\]
which satisfies Assumption \ref{asmp:shypPDE_SLhyperg} with $\eta(x) = {2 \over 1+x}$. The function 
\[
w_\lambda(x) = \begin{cases}
{1 \over 1+x} [\cos(\tau x) + {1 \over \tau} \sin(\tau x)], & \tau > 0 \\
1, & \tau = 0
\end{cases} \qquad (\lambda = \tau^2)
\]
is the solution of the boundary value problem \eqref{eq:shypPDE_ode_wsol}, thus the $\mathcal{L}$-transform can be expressed as a sum of cosine and sine Fourier transforms. A straightforward computation \cite[Example 4.10]{zeuner1992} shows that the product formula $w_\lambda(x) \, w_\lambda(y) = \int_{[a,b)} w_\lambda\, d(\delta_x * \delta_y)$ holds for $\delta_x * \delta_y$ defined by
\begin{equation} \label{eq:exampl_squarehypgr}
(\delta_x * \delta_y)(d\xi) = {1 \over 2(1+x)(1+y)}\bigl[(1+|x-y|)\delta_{x-y}(d\xi) + (1+x+y)\delta_{x+y}(d\xi) + (1+\xi)\mathds{1}_{[|x-y|,x+y]}(\xi) d\xi \bigr]
\end{equation}
and therefore (by the uniqueness property, Proposition \ref{prop:shypPDE_Ltransfmeas_props}(b)) the $\mathcal{L}$-convolution is given by \eqref{eq:exampl_squarehypgr}. This example, which was introduced in \cite[Example 4.10]{zeuner1992}, illustrates that, in general, convolutions associated with regular Sturm-Liouville operators have both a discrete and an absolutely continuous component.
\end{example}

Next we present the chief example of a convolution associated with a singular Sturm-Liouville operator:

\begin{example}[Hankel transform] \label{exam:hankelkingman}
Let $\alpha \geq -{1 \over 2}$. The Bessel operator
\[
\ell = - {d \over dx^2} - {2\alpha + 1 \over x} {d \over dx}, \qquad 0 < x < \infty
\]
has coefficients $p(x) = r(x) = x^{2\alpha+1}$. Clearly, Assumption \ref{asmp:shypPDE_SLhyperg} holds with $\eta = 0$. Here the kernel of the $\mathcal{L}$-transform is
\[
w_\lambda(x) = \bm{J}_\alpha(\tau x) := 2^\alpha \Gamma(\alpha+1) (\tau x)^{-\alpha} J_\alpha(\tau x) \qquad (\lambda = \tau^2)
\]
where $J_\alpha$ is the Bessel function of the first kind (this is easily checked using the basic properties of the Bessel function, cf.\ \cite[Chapter 10]{dlmf}). The Sturm-Liouville type transform associated with the Bessel operator is the \emph{Hankel transform}, $(\mathcal{F}h)(\tau) = \int_0^\infty h(x) \, \bm{J}_\alpha(\tau x) \, x^{2\alpha + 1} dx$. It follows from classical integration formulae for the Bessel function \cite[p.\ 411]{watson1944} that $\bm{J}_\alpha(\tau x) \, \bm{J}_\alpha(\tau y) = \int_0^\infty \bm{J}_\alpha(\tau \xi) \, (\delta_x *_\alpha \delta_y)(d\xi)$, where
\[
(\delta_x *_\alpha \delta_y)(d\xi) = {2^{1-2\alpha} \Gamma(\alpha+1) \over \sqrt{\pi} \, \Gamma(\alpha + {1 \over 2})} (xy\xi)^{-2\alpha} \bigl[ (\xi^2 - (x-y)^2) ((x+y)^2 - \xi^2) \bigr]^{\alpha - 1/2} \mathds{1}_{[|x-y|,x+y]}(\xi) \, r(\xi) d\xi
\]
for $x, y > 0$; this convolution is known as the \emph{Hankel convolution} \cite{hirschman1960,cholewinski1965} or \emph{Kingman convolution} \cite{kingman1963,urbanik1988}.

This example has motivated the development of the theory of generalized translation and convolution operators back since the pioneering work of Delsarte \cite{delsarte1938}. It plays a special role in the context of the Sturm-Liouville hypergroups in Example \ref{exam:SLhypergr} below; in particular, it appears as the limit distribution in central limit theorems on hypergroups \cite[Section 7.5]{bloomheyer1994}. Moreover, since the diffusion ($\mathcal{L}$-Lévy) process generated by $\ell$ is the Bessel process --- a fundamental continuous-time stochastic process \cite{borodinsalminen2002}, which in the case $\alpha = {d \over 2} - 1$ ($d \in \mathbb{N}$) can be defined as the radial part of a $d$-dimensional Brownian motion --- the Hankel convolution is a useful tool for the study of the Bessel process, cf.\ e.g.\ \cite{rentzschvoit2000,vanthu2007}.
\end{example}

The Jacobi operator provides another example of a singular Sturm-Liouville operator whose the product formula and convolution can be written in terms of standard special functions.

\begin{example}[Jacobi transform] \label{exam:fourjacobi}
The coefficients $p(x) = r(x) = (\sinh x)^{2\alpha + 1} (\cosh x)^{2\beta + 1}$ ($\alpha \geq \beta \geq -{1 \over 2}$, \, $\alpha \neq {1 \over 2}$) give rise to the Jacobi operator
\[
\ell = - {d \over dx^2} - [(2\alpha + 1) \coth x + (2\beta + 1) \tanh x] {d \over dx}, \qquad 0 < x < \infty.
\]
As in the previous example, Assumption \ref{asmp:shypPDE_SLhyperg} holds with $\eta = 0$. The so-called Jacobi function
\[
w_\lambda(x) = \phi_\tau^{(\alpha,\beta)\!}(x) := {}_2F_1\Bigl(\tfrac{1}{2}(\sigma - i\tau), \tfrac{1}{2}(\sigma + i\tau); \alpha + 1; -(\sinh x)^2\Bigr) \qquad (\sigma = \alpha + \beta + 1,\; \lambda = \tau^2 + \sigma^2)
\]
where ${}_2F_1$ denotes the hypergeometric function \cite[Chapter 15]{dlmf}, can be shown to be the unique solution of the Sturm-Liouville problem \eqref{eq:shypPDE_ode_wsol}. The associated integral transform is the \emph{(Fourier-)Jacobi transform}, $(\mathcal{F}h)(\tau) = \int_0^\infty h(x) \, \phi_\tau^{(\alpha,\beta)\!}(x) \, (\sinh x)^{2\alpha + 1} (\cosh x)^{2\beta + 1} dx$ (this transformation is also known as Olevskii transform, index hypergeometric transform or, in the case $\alpha = \beta$, generalized Mehler-Fock transform \cite{yakubovich2006}). By a deep result of Koornwinder \cite{flenstedjensen1973,koornwinder1984}, the product formula $\phi_\tau^{(\alpha,\beta)\!}(x) \, \phi_\tau^{(\alpha,\beta)\!}(y) = \int_0^\infty \phi_\tau^{(\alpha,\beta)} d(\delta_x *_{\alpha,\beta} \delta_y)$ holds for the \emph{Jacobi convolution}, defined by
\begin{align*}
(\delta_x *_{\alpha,\beta} \delta_y)(d\xi) = \, &  {2^{-2\sigma} \Gamma(\alpha+1) (\cosh x \, \cosh y \, \cosh \xi)^{\alpha - \beta - 1} \over \sqrt{\pi}\, \Gamma(\alpha + {1 \over 2}) (\sinh x \, \sinh y \, \sinh \xi)^{2\alpha}} \times \\
& \times (1-Z^2)^{\alpha - 1/2} {\,}_2F_1\Bigl( \alpha + \beta, \alpha - \beta; \alpha + \tfrac{1}{2}; \tfrac{1}{2}(1-Z) \Bigr) \mathds{1}_{[|x-y|,x+y]}(\xi) r(\xi) d\xi
\end{align*}
where $Z := {(\cosh x)^2 + (\cosh y)^2 + (\cosh \xi)^2 - 1 \over 2\cosh x \, \cosh y \, \cosh \xi}$.

For half-integer values of the parameters $\alpha, \beta$, the Jacobi transform and convolution have various group theoretic interpretations; in particular, they are related with harmonic analysis on rank one Riemannian symmetric spaces \cite{koornwinder1984}. Moreover, a remarkable property of the Jacobi transform is that it admits a positive dual convolution structure, that is, there exists a family $\{\theta_{\tau_1,\tau_2}\}$ of finite positive measures such that the dual product formula $\phi_{\tau_1}^{(\alpha,\beta)\!}(x) \, \phi_{\tau_2}^{(\alpha,\beta)\!}(x) = \int_0^\infty \phi_{\tau_3}^{(\alpha,\beta)\!}(x) \, \theta_{\tau_1,\tau_2}(d\tau_3)$ holds, and this permits the construction of a generalized convolution which trivializes the inverse Jacobi transform \cite{bensalem1994}.
\end{example}

All the examples presented so far belong to the class of Sturm-Liouville hypergroup convolutions which was introduced by Zeuner \cite{zeuner1992} as follows:

\begin{example}[Sturm-Liouville hypergroups] \label{exam:SLhypergr}
Consider a Sturm-Liouville operator on the positive half-line with coefficients $p = r = A$,
\[
\ell = - {d^2 \over dx^2} - {A'(x) \over A(x)} {d \over dx}, \qquad 0 < x < \infty,
\]
where the function $A$ satisfies the following conditions:
\begin{enumerate}[itemsep=0pt,topsep=4pt]
\item[\textbf{SL0}] $A \in \mathrm{C}[0,\infty) \cap \mathrm{C}^1(0,\infty)$ and $A(x) > 0$ for $x > 0$.
\item[\textbf{SL1}] One of the following assertions holds: \vspace{-0.8ex}
\begin{enumerate}[itemsep=0pt]
\item[\textbf{SL1.1}] $A(0) = 0$ and ${A'(x) \over A(x)} = {\alpha_0 \over x} + \alpha_1(x)$ for $x$ in a neighbourhood of $0$, where $\alpha_0 > 0$ and $\alpha_1 \in \mathrm{C}^\infty(\mathbb{R})$ is an odd function;
\item[\textbf{SL1.2}] $A(0) > 0$ and $A \in \mathrm{C}^1[0,\infty)$. \vspace{-0.8ex}
\end{enumerate}
\item[\textbf{SL2}] There exists $\eta \in \mathrm{C}^1[0,\infty)$ such that $\eta \geq 0$, $\bm{\phi}_\eta \geq 0$ and the functions $\bm{\phi}_\eta$, $\bm{\psi}_\eta$ are both decreasing on $(0,\infty)$ ($\bm{\phi}_\eta$, $\bm{\psi}_\eta$ are defined as in Assumption \ref{asmp:shypPDE_SLhyperg}).
\end{enumerate}
The last condition ensures that $A$ satisfies Assumption \ref{asmp:shypPDE_SLhyperg}, hence this is a particular case of the general family of Sturm-Liouville operators considered in the previous sections. It was proved by Zeuner \cite{zeuner1992} that the convolution measure algebra $(\mathcal{M}_\mathbb{C}[0,\infty),*)$ is a commutative hypergroup with identity involution; this means that the Banach algebra property of Proposition \ref{prop:shypPDE_conv_Mbanachalg} and properties (b)--(c) of Corollary \ref{cor:shypPDEconv_basicprops} hold, as well as the following axioms:
\begin{itemize}[itemsep=-1pt,topsep=4pt]
\item $(x,y) \mapsto \supp(\delta_x * \delta_y)$ is continuous from $[0,\infty) \times [0,\infty)$ into the space of compact subsets of $[0,\infty)$ (endowed with the Michael topology, see \cite{jewett1975});
\item $0 \in \supp(\delta_x * \delta_y)$ if and only if $x = y$.
\end{itemize}
Observe that the Sturm-Liouville operator $\ell = -{d^2 \over dx^2} - {A' \over A} {d \over dx}$ is either singular or regular, depending on whether the function $A$ satisfies condition SL1.1 or SL1.2. In any event, the associated hyperbolic equation $\ell_x f = \ell_y f$ is uniformly hyperbolic on $[0,\infty)^2$. The construction of the product formula and convolution presented in the previous sections generalizes that of Zeuner because it is also applicable to parabolically degenerate operators.
\end{example}

The next example shows that the two hypergroup axioms on the (compact) support of $\delta_x * \delta_y$ are generally false for operators associated with degenerate hyperbolic equations:

\begin{example}[Index Whittaker transform] \label{exam:whittaker}
The choice $p(x) = x^{2-2\alpha} e^{-1/x}$ and $r(x) = x^{-2\alpha} e^{-1/x}$, with $\alpha < {1 \over 2}$, leads to the normalized Whittaker operator
\[
\ell = - x^2 {d^2 \over dx^2} - (1+2(1-\alpha)x) {d \over dx}, \qquad 0 < x < \infty.
\]
The standard form of this differential operator (Remark \ref{rmk:shypPDE_tildeell}) is $\widetilde{\ell} = - {d^2 \over dz^2} - (e^{-z} + 1 - 2\alpha){d \over dz}$, where $z = \log x \in \mathbb{R}$, and it is apparent that Assumption \ref{asmp:shypPDE_SLhyperg} holds with $\eta = 0$. As pointed out in Section \ref{sec:hypPDE}, the fact that the operator $\widetilde{\ell}$ is defined on the whole real line means that the hyperbolic partial differential equation associated with the normalized Whittaker operator has a non-removable parabolic degeneracy at the initial line. The unique solution of the boundary value problem \eqref{eq:shypPDE_ode_wsol} turns out to be given by
\[
w_\lambda(x) = \bm{W}_{\!\!\alpha,i\tau}(x) := x^\alpha e^{1 \over 2x} W_{\alpha, i\tau}(\tfrac{1}{x}) \qquad \bigl( \lambda = \tau^2 + (\tfrac{1}{2} - \alpha)^2 \bigr)
\]
where $W_{\alpha,i\tau}$ is the Whittaker function of the second kind of parameters $\alpha$ and $i\tau$ \cite[Chapter 13]{dlmf}. The eigenfunction expansion of the normalized Whittaker operator yields the \emph{index Whittaker transform} \cite{srivastava1998,sousayakubovich2018} $(\mathcal{F}h)(\tau) = \int_0^\infty  h(x) \bm{W}_{\!\!\alpha,i\tau}(x) x^{-2\alpha} e^{-1/x} dx$. The product formula for the kernel $\bm{W}_{\!\!\alpha,i\tau}$ has recently been established by the authors \cite{sousaetal2018a,sousaetal2018b} using techniques from classical analysis and known facts in the theory of special functions; it is given by $\bm{W}_{\!\!\alpha,i\tau}(x) \bm{W}_{\!\!\alpha,i\tau}(y) = \int_0^\infty \bm{W}_{\!\!\alpha,i\tau} \, d(\delta_x *_\alpha \delta_y)$, where $*_\alpha$ is the \emph{Whittaker convolution}, defined by
\[
(\delta_x *_\alpha \delta_y)(d\xi) = {2^{-1-\alpha} \over \sqrt{\pi}} (xy\xi)^{-{1\over 2}+\alpha} \exp\Bigl( {1 \over x} + {1 \over y} + {1 \over \xi} - {(x+y+\xi)^2 \over 8xy\xi} \Bigr) D_{2\alpha}\Bigl( {x+y+\xi \over \smash{\sqrt{2xy\xi}}} \Bigr) r(\xi) d\xi
\]
for $x,y > 0$, with $D_\mu$ denoting the parabolic cylinder function \cite[Chapter VIII]{erdelyiII1953}. Notice in particular that $\supp(\delta_x *_\alpha \delta_y) = [0,\infty)$ for every $x, y > 0$. 

The particular case $\alpha = 0$ is worthy of special mention, because in this case the index Whittaker transform reduces to $(\mathcal{F}h)(\tau) = \pi^{-1/2} \int_0^\infty h(x) K_{i\tau}({1 \over 2x}) x^{-1/2} e^{-{1 \over 2x}} dx$, which is (a normalized form of) the Kontorovich-Lebedev transform; here $K_{i\tau}$ is the modified Bessel function of the second kind with parameter $i\tau$ \cite[Chapter 10]{dlmf}. The Kontorovich-Lebedev transform plays a central role in the theory of index type integral transforms \cite{yakubovich1996}. The Whittaker convolution of parameter $\alpha = 0$, which can be written in the simplified form
\[
(\delta_x *_0 \delta_y)(d\xi) = {1 \over 2\sqrt{\pi x y \xi}} \exp\Bigl( {1 \over x} + {1 \over y} - {(x+y+\xi)^2 \over 4xy\xi} \Bigr) d\xi,
\]
is identical (up to an elementary change of variables) to the Kontorovich-Lebedev convolution, which was introduced by Kakichev in \cite{kakichev1967} and has been extensively studied, cf.\ \cite{yakubovich1996} and references therein.
\end{example}

Our final example illustrates that the (degenerate) hyperbolic equation approach allows us to generalize the results on the Whittaker product formula and convolution to a much larger class of degenerate operators:

\begin{example}
Let $\bm{\zeta} \in \mathrm{C}^1(0,\infty)$ be a nonnegative decreasing function such that $\int_1^\infty \bm{\zeta}(y) {dy \over y} = \infty$, and let $\kappa > 0$. The differential expression
\[
\ell = - x^2 {d^2 \over dx^2} - \bigl[\kappa + x \bigl(1 + \bm{\zeta}(x)\bigr)\bigr] {d \over dx}, \qquad 0 < x < \infty
\]
is a particular case of \eqref{eq:shypPDE_elldiffexpr}, obtained by considering $p(x) = x e^{-\kappa/x + I_{\bm{\zeta}}(x)}$ and $r(x) = {1 \over x} e^{-\kappa/x + I_{\bm{\zeta}}(x)}$, where $I_{\bm{\zeta}}(x) = \int_1^x \bm{\zeta}(y) {dy \over y}$. (If $\kappa = 1$ and $\bm{\zeta}(x) = 1-2\alpha > 0$, we recover the normalized Whittaker operator from Example \ref{exam:whittaker}.) The change of variable $z = \log x \in \mathbb{R}$ transforms $\ell$ into the standard form $\widetilde{\ell} = - {d^2 \over dz^2} - {A'(z) \over A(z)}{d \over dz}$, where ${A'(z) \over A(z)} = \kappa e^{-\kappa z} + \bm{\zeta}(e^z)$. It is clear that $\ell$ satisfies Assumption \ref{asmp:shypPDE_SLhyperg} with $\eta = 0$, and the additional assumption $\lim_{x \uparrow b} p(x)r(x) = \infty$ holds because $I_{\bm{\zeta}}(\infty) = \infty$. Therefore, all the results in the previous sections hold for the  Sturm-Liouville operator $\ell$. This shows that the class of Sturm-Liouville operators for which one can construct a positivity-preserving convolution structure includes irregular operators which are simultaneously degenerate (in the sense that the associated hyperbolic equation is parabolic at the initial line) and singular (in the sense that the first order coefficient is unbounded near the left endpoint).
\end{example}

\section*{Acknowledgements}

The first and third authors were partly supported by CMUP (UID/MAT/00144/2019), which is funded by Fundação para a Ciência e a Tecnologia (FCT) (Portugal) with national (MCTES) and European structural funds through the programs FEDER, under the partnership agreement PT2020, and Project STRIDE -- NORTE-01-0145-FEDER-000033, funded by ERDF -- NORTE 2020. The first author was also supported by the grant PD/BD/135281/2017, under the FCT PhD Programme UC|UP MATH PhD Program. The second author was partly supported by the project CEMAPRE -- UID/MULTI/00491/2013 financed by FCT/MCTES through national funds.

\renewcommand{\bibname}{References} 
\begin{small}

\end{small}

\end{document}